\documentclass[twoside,11pt]{article}%{{{1

\usepackage{amsthm,amsmath,amsfonts,amssymb,bm,bbm,mathtools} 
\usepackage{graphicx}
\usepackage{tocloft} % For controling spacing in table of contents

\usepackage{datetime} \usdate
\usepackage{currfile} 
\usepackage{color}
\usepackage{hyperref}

\date{October 19, 2015}
\newif\ifdraft
\ifdraft
\date{DRAFT: \today}
\fi

\newcommand{\myrunningheads}
{\ifdraft
\pagestyle{myheadings}
\markboth
{Filename: \currfilename \quad DRAFT \mmddyyyydate\today\quad \currenttime }
{Filename: \currfilename \quad DRAFT \mmddyyyydate\today\quad \currenttime }
\else
\pagestyle{myheadings}
\markboth
{Coagulation-fragmentation model for animal group-size statistics}
{P. Degond, J.-G. Liu and R. L. Pego}
\fi
}
\myrunningheads

\usepackage{pdfsync,hyperref}

\newcommand{\nwc}{\newcommand}
\newcommand{\hide}[1]{} % Hides whatever is in braces: replace by {#1} to unhide
\nwc{\PhiNew}{\Phi_{\star}}

\nwc{\iav}{{i_{\rm av}}}
\nwc{\xav}{{x_{\rm av}}}
\nwc{\feq}{{f_{\rm eq}}}
\nwc{\Feq}{{F_{\rm eq}}}
\nwc{\fheq}{{f^h_{\rm eq}}}
\nwc{\fheqi}{{f^h_{\rm eq\,\it i}}}
\nwc{\Fheq}{{F^h_{\rm eq}}}

\nwc{\finit}{f_{\scriptsize\rm in}}
\nwc{\Finit}{F_{\scriptsize\rm in}}

\nwc{\mui}{{\mu\:\! i}}  % subscript mu with half-thin space

\renewcommand{\epsilon}{\varepsilon}
\nwc{\eps}{\epsilon}

\nwc{\ip}[1]{\langle #1 \rangle}
\nwc{\qref}[1]{(\ref{#1})}

\nwc{\rplus}{{{\mathbb R}_+}}
\nwc{\D}{\partial}
\nwc{\inv}{^{-1}}
\nwc{\ubar}{\underline}
\nwc{\uu}{U}

\newcommand{\supn}{^{(n)}}
\nwc{\supk}{^{(k)}}
\nwc{\supj}{^{(j)}}
\nwc{\wkto}{\xrightarrow{w}}
\nwc{\vto}{\xrightarrow{v}}
\nwc{\nto}{\xrightarrow{n}}

\nwc{\np}{{n+1}}
\nwc{\one}{\mathbbm{1}}
\newcommand{\R}{\mathbb{R}}
\newcommand{\C}{\mathbb{C}}

\newcommand{\N}{\mathbb{N}^*}

\newcommand{\calM}{\mathcal{M}}
\newcommand{\calL}{\mathcal{L}}

\def\theequation{\thesection.\arabic{equation}}

\def\paragraph#1{{\bf #1\ }}

\newtheorem{lemma}{Lemma}[section]  

\newtheorem{theorem}[lemma]{Theorem}

\newtheorem{definition}[lemma]{Definition}
\newtheorem{proposition}[lemma]{Proposition}

\theoremstyle{definition}  % The following environments do not use italicS
\newtheorem{remark}{Remark}[section]

\numberwithin{equation}{section}

\makeatletter
\@mparswitchfalse%Put all sidenotes on the right.
\makeatother

\makeatletter
\newcommand{\eqlab}[1]{\leavevmode\hfill\refstepcounter{equation}\label{#1}\textup{\tagform@{\theequation}}}
\makeatother

\newcommand{\warning}[1]{\typeout{}\typeout{WARNING: #1 at line \the\inputlineno}\typeout{}}
\newenvironment{todo}[1][TODO]{%
    \ifdraft\else\warning{TODO still present in final version}\fi
    \MakeFramed{\advance\hsize-\width \FrameRestore}\textbf{#1. }}%
    {\endMakeFramed}
    {\end{todo}}

\DeclareMathOperator{\im}{{\rm Im}}

\def\Box{\leavevmode\vbox{\hrule
     \hbox{\vrule\kern4pt\vbox{\kern4pt}%
           \vrule}\hrule}}
\def\blackbox{\leavevmode\vrule height 5pt width 4pt depth 0pt\relax}
\def\endproof{\null\hfill {$\blackbox$}\bigskip}

\newcounter{appendix}
\def\appendix{\advance\c@appendix by 1
   \def\thesection{\Alph{section}}
   \ifnum\c@appendix=1 \setcounter{section}{-1} \fi
   \@startsection {section}{1}{\z@}{-3.5ex plus -1ex minus 
   -.2ex}{2.3ex plus .2ex}{\Large\bf}}

\makeatletter
\def\@part[#1]#2{%
  \ifnum \c@secnumdepth >-2\relax
    \refstepcounter{part}%
    \addcontentsline{toc}{part}{\thepart\hspace{1em}#1}%
  \else
    \addcontentsline{toc}{part}{#1}%
  \fi
  \markboth{}{}%
  {\centering
   \interlinepenalty \@M
   \normalfont
   \ifnum \c@secnumdepth >-2\relax
     \LARGE\bfseries \thepart\ \ % 
   \fi
   #2\par}
  }
\makeatother
\date{} 
\begin{document}

\title
{Coagulation-fragmentation model  for \\ animal group-size statistics} 

\author{Pierre Degond $^{(1)}$, Jian-Guo Liu$^{(2)}$, Robert L. Pego$^{(3)}$}  

\maketitle

\vspace{-0.2 cm}

\begin{center}
1-Department of Mathematics\\
Imperial College London,
London SW7 2AZ, UK\\
email:pdegond@imperial.ac.uk
%1-Universit\'e de Toulouse; UPS, INSA, UT1, UTM ;\\ 
%Institut de Math\'ematiques de Toulouse ; \\
%F-31062 Toulouse, France. \\
%2-CNRS; Institut de Math\'ematiques de Toulouse UMR 5219 ;\\ 
%F-31062 Toulouse, France.\\
%email: pierre.degond@math.univ-toulouse.fr
\end{center}

\begin{center}
2-Department of Physics and Department of Mathematics\\
Duke University,
Durham, NC 27708, USA\\
email: jliu@phy.duke.edu
\end{center}

\begin{center}
3-Department of Mathematics
and Center for Nonlinear Analysis\\
Carnegie Mellon University,
Pittsburgh, Pennsylvania, PA 12513, USA\\
email: rpego@cmu.edu
\end{center}

\vspace{-0.2 cm}
\begin{abstract}
We study coagulation-fragmentation equations inspired by a simple model
proposed in fisheries science to explain data for the size distribution of
schools of pelagic fish.  Although the equations lack detailed balance
and admit no $H$-theorem, we are able to develop a rather complete
description of equilibrium profiles and large-time behavior, 
based on recent developments in complex function theory for Bernstein and Pick functions. 
In the large-population continuum limit, 
a scaling-invariant regime is reached in which all equilibria are determined by a
single scaling profile. This universal profile exhibits power-law behavior crossing over from
exponent $-\frac23$ for small size to $-\frac32$ for large size, with an exponential cut-off. 
\end{abstract}

%\medskip
%\noindent
%{\bf Acknowledgements:} text of the acknowledgements

\medskip
\noindent
{\bf Key words: } Detailed balance, fish schools, Bernstein functions, complete monotonicity, Fuss-Catalan sequences, convergence to equilibrium.

\medskip
\noindent
{\bf AMS Subject classification: } 45J05, 70F45, 92D50, 37L15, 44A10, 35Q99.
\vskip 0.4cm

\pagebreak
    
\setlength{\cftbeforesecskip}{4pt}
\setlength{\cftbeforepartskip}{9pt}
\renewcommand{\cftsecfont}{\normalfont}
\tableofcontents

%%%%%%%%%%%%%%%%%%%%%%%%%%%%%%%%%%%%%%%%%%%%%%%%%%%%%%%%%%%%%%%%%%%%%%%%%%%%%%%%%%%%%%%%%%%%%%%%
%%%%%%%%%%%%%%%%%%%%%%%%%%%%%%%%%%%%%%%%%%%%%%%%%%%%%%%%%%%%%%%%%%%%%%%%%%%%%%%%%%%%%%%%%%%%%%%%
%%%%%%%%%%%%%%%%%%%%%%%%%%%%%%%%%%%%%%%%%%%%%%%%%%%%%%%%%%%%%%%%%%%%%%%%%%%%%%%%%%%%%%%%%%%%%%%%
%%%%%%%%%%%%%%%%%%%%%%%%%%%%%%%%%%%%%%%%%%%%%%%%%%%%%%%%%%%%%%%%%%%%%%%%%%%%%%%%%%%%%%%%%%%%%%%%

%\typeout{Line Width: \the\linewidth}

%% Text of document
\vfil\pagebreak

% !TEX root = DLPmain.tex

%%%%%%%%%%%%%%%%%%%%%%%

\setcounter{equation}{0}
\section{Introduction}
\label{intro}

A variety of methods have been used to account for the observed 
statistics of animal group size in population ecology. 
The present work focuses on coagulation-fragmentation equations
that are motivated by studies of 
Niwa \cite{Niwa-CMA1996,Niwa-JTB1998,Niwa-JTB2003,Niwa-JTB2004}
and subsequent work by Ma et al.~\cite{Ma_etal_JTB11},
to explain observations in fisheries science
that concern the size distribution of schools of pelagic fish,
which roam in the mid-ocean.

In \cite{Niwa-JTB2003} and \cite{Niwa-JTB2004}, Niwa reached the striking 
conclusion that a large amount of observational data 
indicates that the fish school-size distribution $(f_i)_{i\ge1}$
is well described by a scaling relation of the form
\begin{equation}
    f_i \ \propto \ \frac1\iav \,\Phi\left(\frac i{\iav}\right) \ ,
\qquad \iav = \frac{\sum_i i^2 f_i} {\sum_i i f_i}\ .
\end{equation}
The scaling factor $\iav$
is the expected group size averaged over individuals,
and the universal profile $\Phi$ is highly non-Gaussian---instead it is a power law 
with an exponential cutoff at large size. Specifically Niwa proposed that
\begin{equation}
    \Phi(x) = x\inv \exp\left( -x+ \frac12 x e^{-x}\right).
\label{Phi-Niwas}
\end{equation}
%\begin{equation}
%f_i \sim \frac1i e^{-i/\bar \imath} \ ,
%\end{equation}
Note that there are no fitting parameters.

Niwa discussed how this description might be justified for pelagic fish
in a couple of different ways.  Of particular interest for us here is the fact that
he performed kinetic Monte-Carlo simulations of a 
coagulation-fragmentation or merging-splitting process
with the following features:
\begin{itemize}
\item The ocean is modeled as a discrete set of sites that fish schools may occupy.
\item Schools jump to a randomly chosen site at discrete time steps.
 \item Two schools arriving at the same site merge.
 \item Any school may split in two with fixed probability per time step, 
\begin{itemize}
\item independent of the school size $i$,
\item with uniform likelihood among the $i-1$ splitting outcomes
\[
(1,i-1), (2,i-2),\ldots,(i-1,1)\ .
\]
\end{itemize}
\end{itemize}

Niwa's model is simple and compelling. 
It corresponds to mean-field coagulation-fragmentation
equations with a constant rate kernel for coagulation and constant overall fragmentation rate; 
see equations \qref{eq:CF3_disc}-\qref{eq:CF5_disc} and \qref{eq:rates_niwa_disc} 
in section 2 below. Such equations were explicitly written and 
studied in a time-discrete form by Ma et al.\ in \cite{Ma_etal_JTB11}.
Yet the existing mathematical
theory of coagulation-fragmentation equations reveals little about
the nature of their equilibria and the dynamical behavior of solutions.  

The reason for this dearth of theory is that the
 existing results that concern equilibria and long-time behavior 
almost all deal with systems that admit equilibria in detailed balance, 
with equal rates of merging and splitting for each reaction taking 
clusters of sizes $i,j$ to one of size $i+j$.
For modeling animal group sizes in particular, Gueron and Levin \cite{GL1995}
discussed several coagulation-fragmentation models for continuous size distributions 
with explicit formulae for equilibria having detailed balance.
However, Niwa argued explicitly in \cite{Niwa-JTB2003} that the observational data 
for pelagic fish is inconsistent with these models.

%Due to the lack of general methods for studying dynamics in wide classes of  
%coagulation-fragmentation equations, it seems necessary to study special models
%with any method available. 

One of the principal contributions of the present paper is a 
demonstration that the coagulation-fragmentation system 
achieves a scaling-invariant regime in the large-population, continuum limit.
In this limit, all equilibrium size distributions $f_{\rm eq}(x)$, $x\in(0,\infty)$, 
are described rigorously in terms of a single scaling profile $\PhiNew$, with
%[Scaling to handle general $p,q$: ]
\begin{equation} \label{eq:feqPhi}
\feq(x) \ \propto \ %
\frac1{\xav} \,\PhiNew\left(\frac{x}{\xav}\right)  ,
\qquad
\xav = \frac
{\int_0^\infty x^2 \feq(x)\,dx}
{\int_0^\infty x \feq(x)\,dx} .
\end{equation}
%
%\frac{p}{qm_1} f_\infty(x/m_1), \qquad
%N = \frac{\int x^2 f_\infty(x/m_1)\,dx}{\int x f_\infty(x/m_1)\,dx} = m_{2,\infty} m_1. 
%\]
The profile $\PhiNew$ admits a series representation described in 
subsection~\ref{sec:seriesC} below, 
and we provide considerable qualitative information regarding its shape. 
In particular,
\begin{equation}
\PhiNew(x) = g(x) \, e^{- \frac{8}{9} x}, 
\label{Phi-ours}
\end{equation}
where $g$ is a \textit{completely monotone} function
(infinitely differentiable with derivatives that alternate in sign), 
having the following asymptotic behavior:
\begin{eqnarray}
   g(x) &\sim& 6^{1/3} \frac{x^{-2/3}}{\Gamma(1/3)} \, , \qquad \mbox{when} \quad x \to 0, 
\label{eq:profile2} \\
g(x) &\sim& \frac9{8\sqrt6} \frac{x^{-3/2}} {\Gamma(1/2)}  \,, \qquad \mbox{when} \quad x \to \infty .
\label{eq:profile3} 
\end{eqnarray}
%\begin{equation} \label{eq:gprofile}
%g(x)\sim \frac{(6x)^{-2/3}}{\Gamma(1/3)} \quad\mbox{as $x\to0$,}
%\qquad
%g(x)\sim \frac98 \frac{(6x)^{-3/2}}{\Gamma(1/2)} \quad\mbox{as $x\to\infty$.}
%\end{equation}
%\begin{eqnarray}
%    &&g(x) \sim \frac{1}{\Gamma(1/3)} \, (6x)^{-2/3}, \qquad \mbox{when} \quad x \to 0, 
%\label{eq:profile2} \\
%&&g(x) \sim \frac{9}{8\Gamma(1/2)} \, (6x)^{-3/2}, \qquad \mbox{when} \quad x \to \infty .
%\label{eq:profile3} 
%\end{eqnarray}
Moreover, $\PhiNew$ is a proper probability density, satisfying %[Remark to omit: $\Phi(x) = 6 f_\infty(6x)$]
\begin{equation}
1 = \int_0^\infty \PhiNew(x)\,dx = 6 \int_0^\infty x \PhiNew(x)\,dx =
6 \int_0^\infty x^2 \PhiNew(x)\,dx.
\label{eq:profile4}
\end{equation}
\begin{figure}[b!t]
    \begin{center}
        \includegraphics[trim=0 430 0 40, clip,width=\textwidth]{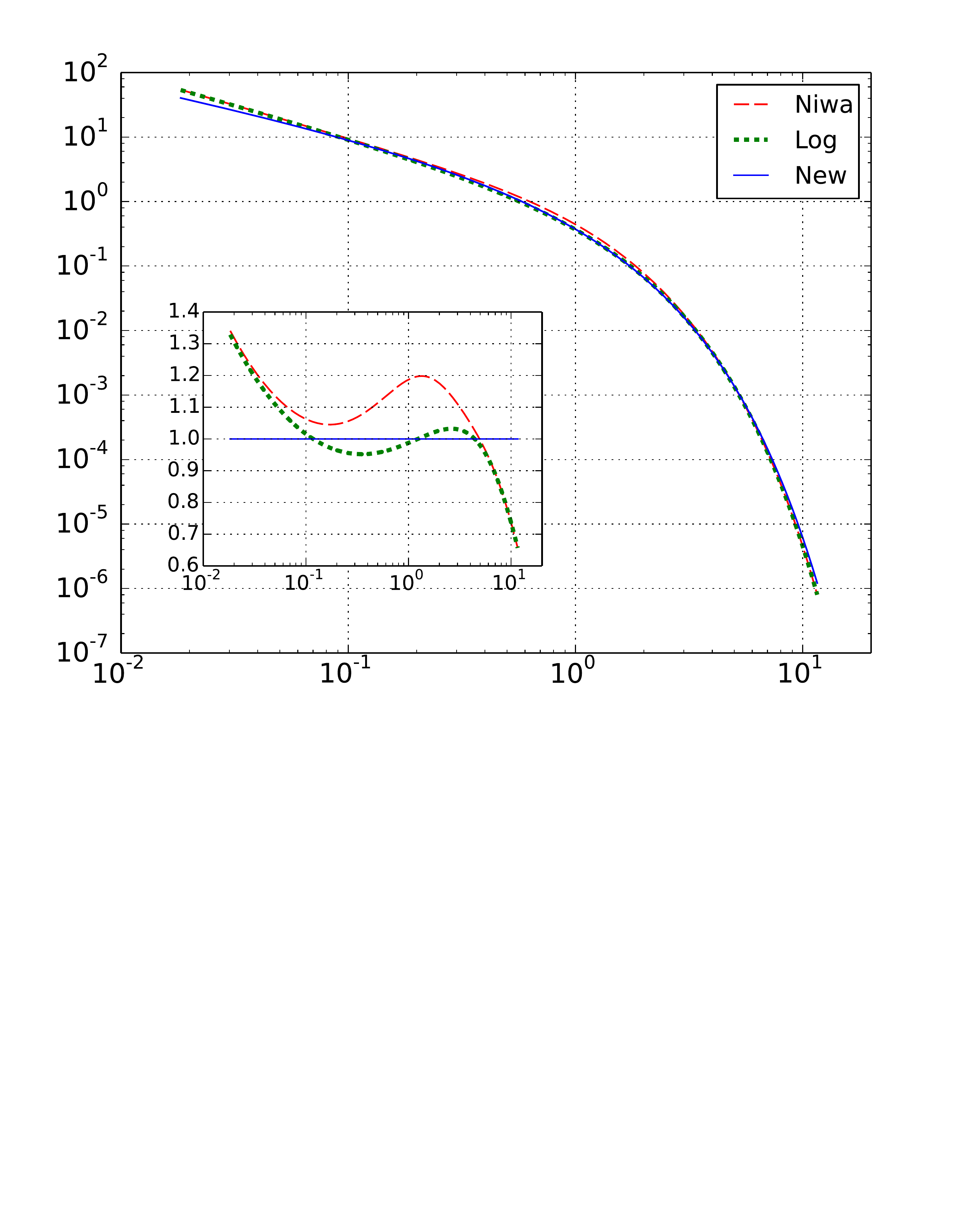}
\caption{Log-log plots of $\Phi(x)$ {vs.}~$x$ for 
Niwa's, logarithmic, and new distribution profiles.
Inset: Ratios $\Phi/\Phi_{\star}$ for all three cases.}
    \label{fig:compare}
    \end{center}
\end{figure}
The crossover in the power-law behavior exhibited by the exponential prefactor 
$g(x)$, from $x^{-2/3}$ for small $x$ to $x^{-3/2}$ for large $x$, 
suggests that it could be difficult in practice to distinguish 
the profile in \qref{Phi-ours} from the expression in \qref{Phi-Niwas}.
Indeed, in Figure~1 we compare 
Niwa's profile in \qref{Phi-Niwas}, as well as the simple logarithmic distribution profile
\[
    \Phi(x)=x\inv \exp(-x) \,,
\]
to the new profile $\PhiNew$ in \qref{Phi-ours}, computed 
using \qref{Phi-fstar} and 45 terms from the power series in subsection~\ref{sec:seriesC}. 
As reported in \cite{Ma_etal_JTB11}, Niwa noted that as far as the 
quality of data fitting is concerned, 
the nested exponential term in \qref{Phi-Niwas} makes little difference, so 
the simpler logarithmic profile serves about as well.
If one compares Figure~1 here with Figure~5 in \cite{Niwa-JTB2003},
this also appears to be the case with the new profile $\PhiNew$, 
despite the difference in power-law exponents for small $x$ and the 
the slightly slower exponential decay rate in the tail of $\PhiNew$ as compared to 
the other cases ($e^{-8x/9}$ vs. $e^{-x}$).
In the `shoulder' region of the log-log plot 
(which covers the bulk of the empirical data in Figure~5 of \cite{Niwa-JTB2003}),
the new profile differs by only a few percent from the logarithmic profile
and by less than 20 percent from the profile in \qref{Phi-Niwas}.

In addition to this description of equilibrium, we develop a rather complete
theory of dynamics in the continuum limit, 
for weak solutions whose initial data are finite measures on $(0,\infty)$.  
We establish convergence to equilibrium for all solutions that
correspond to finite total population (finite first moment).  Furthermore, for initial
data with infinite first moment, solutions converge to zero in a weak sense,
meaning that the population concentrates in clusters whose size grows without
bound as $t\to\infty$.

Previous mathematical studies of equilibria and  dynamical behavior in 
coagulation-fragmentation models 
include work of Aizenman and Bak \cite{AB1979}, 
Carr \cite{Carr1992}, Carr and da Costa \cite{CDC1994}, 
Lauren\c{c}ot and Mischler \cite{LM2003}, and Ca\~nizo \cite{C-JSP2007}, 
as well as a substantial literature related to Becker-Doering equations
(which take into account only clustering events that involve the 
gain or loss of a single individual). 
These works all concern models having 
detailed balance, and rely on some form of $H$ theorem, or entropy/entropy-dissipation arguments. 
For models without detailed balance, there is work of 
Fournier and Mischler \cite{FM2004}, concerning initial data near equilibrium
in discrete systems, and a recent study by Laurencot and van Roessel \cite{LR2015}
of a model with a multiplicative coagulation rate kernel in critical balance with
a fragmentation mechanism that produces an infinite number of fragments.

Arguments involving entropy are not available for the models that we need to treat here.  Instead, it turns out to be possible to use methods from complex function theory related to the Laplace transform. Such methods have been used to advantage to analyze the dynamics of pure coagulation equations with special rate kernels \cite{MP2004,MP2008}; also see \cite{LR2010,LR2015}. 

Specifically what is relevant for the present work is the
theory of {Bernstein functions}, as developed in the book of 
Schilling et al.~ \cite{Schilling_etal_Bernstein}.  
In terms of the ``Bernstein transform,''  the coagulation-fragmentation 
equation in the continuum limit transforms to a nonlocal 
integro-differential equation
which turns out to permit a detailed analysis of equilibria and long-time dynamics.
 
Unfortunately, the Bernstein transform does not appear to produce a tractable form
for the discrete-size coagulation-fragmentation equations coming 
from \cite{Ma_etal_JTB11} that correspond to Niwa's merging-splitting process with the 
features described above. We have found, however, that a simple change
in the fragmentation mechanism allows one to reduce
 the Bernstein transform \textit{exactly} to the same equation as obtained for
 the continuum limit.  
 One only needs to change the splitting rule
 to assign uniform likelihood among the $i+1$ splitting outcomes
 \[
 (0,i), (1,i-1),\ldots, (i,0).
 \]
In the extreme cases, of course, no splitting actually happens. This change effectively
slows the fragmentation rate for small groups. However, the analysis becomes remarkably simpler, as we will see below.

In particular, for this discrete-size model, we can characterize its equilibrium 
distributions (which depend now on total population) in terms of completely monotone sequences with exponential cutoff. Whenever the total population is initially finite,
every solution converges to equilibrium strongly with respect to a size-weighted norm.
And for infinite total population, again the population concentrates in ever-larger 
clusters---the size distribution converges to zero pointwise
while the zeroth moment goes to a nonzero constant.

The plan of the paper is as follows. 
In section 2 we describe the coagulation-fragmentation models under study
in both the discrete-size setting (Model D) and continuous-size setting (Model C).
A summary of results from the theory of Bernstein functions appears in section 3.
Our analysis of Model C is carried out in sections 4--9 (Part I),
and Model D is treated in sections 10--13 (Part II).
Lastly, in Part III (sections 14--16) we relate Model D to a discretization of Model C, 
and prove a discrete-to-continuum limit theorem.

%%%%%%%%%%%%%%%%%%%%%%%%%%%%%%%%%%%%%%%%%%%%%%%%%%%%%%%%%%%%%%%%%%%%%%%%%%%%%%%%%%%%%%%%%%%%%%%%
%%%%%%%%%%%%%%%%%%%%%%%%%%%%%%%%%%%%%%%%%%%%%%%%%%%%%%%%%%%%%%%%%%%%%%%%%%%%%%%%%%%%%%%%%%%%%%%%
%%%%%%%%%%%%%%%%%%%%%%%%%%%%%%%%%%%%%%%%%%%%%%%%%%%%%%%%%%%%%%%%%%%%%%%%%%%%%%%%%%%%%%%%%%%%%%%%
%%%%%%%%%%%%%%%%%%%%%%%%%%%%%%%%%%%%%%%%%%%%%%%%%%%%%%%%%%%%%%%%%%%%%%%%%%%%%%%%%%%%%%%%%%%%%%%%
\setcounter{equation}{0}
\section{Coagulation-fragmentation Models D and C}
\label{sec:CF_general}

%\subsection{Discrete case}
%\label{subsec:CF_discrete}

In this section we describe the coagulation-fragmentation 
mean-field rate equations that  model Niwa's merging-splitting simulations, 
and the corresponding equations 
for both discrete-size and continuous-size distributions 
that we focus upon in this paper.

\subsection{Discrete-size distributions}
We begin with a general description of coagulation-fragmentation 
equations for a system
consisting of clusters $(i)$ having discrete sizes $i \in {\N}=\{1,2,\ldots\}$. 
Clusters can merge or split according to the following reactions:
\begin{eqnarray*}
&&\hspace{-1cm}
(i) + (j) \, \stackrel{a_{i,j}}{\longrightarrow} \, (i+j) \quad \mbox{ (binary coagulation)}, \\
&&\hspace{-1cm}
(i) + (j) \, \stackrel{b_{i,j}}{\longleftarrow} \, (i+j) \quad \mbox{ (binary fragmentation)}. 
\end{eqnarray*}
Here $a_{i,j}$ is the coagulation rate (i.e., the probability that clusters $(i)$ and $(j)$ 
with (unordered) respective sizes $i$ and $j$ merge into the cluster $(i+j)$ of size $i+j$ per unit of time) 
and $b_{i,j}$ is the fragmentation rate (i.e., the probability that a cluster $(i+j)$ 
splits into two clusters $(i)$ and $(j)$ per unit of time). Both $a_{i,j}$ and $b_{i,j}$ are assumed nonnegative, 
and symmetric in $i,j$. 

We focus on a statistical description of this system in terms of the number density $f_i(t)$ 
of clusters of size $i$ at time $t$. 
The size distribution $f(t)=(f_i(t))_{i\in\N}$ evolves according to the discrete coagulation-fragmentation equations,
written in strong form as follows:
\begin{eqnarray}
&&\hspace{-1.5cm}
\frac{\partial f_i}{\partial t}(t) = Q_a(f)_i(t) + Q_b(f)_i(t) ,
\label{eq:CF3_disc}\\
&&\hspace{-1.5cm}
Q_a(f)_i(t) = \frac{1}{2} \sum_{j=1}^{i-1} a_{j , i-j}\, f_j(t)  \, f_{i-j}(t) - \sum_{j=1}^\infty a_{i,j} \, f_i(t) \, f_j(t) , 
\label{eq:CF4_disc} \\
&&\hspace{-1.5cm}
Q_b(f)_i(t) = \sum_{j=1}^\infty b_{i,j} \, f_{i+j}(t) - \frac{1}{2} \sum_{j=1}^{i-1} b_{j , i-j} \, f_i(t) \ .
\label{eq:CF5_disc}
\end{eqnarray}
Here, the terms in $Q_a(f)_i(t)$ (resp.\ in $Q_b(f)_i(t)$) account for gain and loss of clusters of size $i$ 
due to aggregation/coagulation (resp.\ breakup/frag\-mentation).
It is often useful to write this system in weak form, requiring that for any 
suitable test function $\varphi_i$ (in a class to be specified later), 
\begin{eqnarray}
&&\hspace{-1cm} 
\frac{d}{dt} \sum_{i=1}^\infty  \varphi_i \, f_i(t) =
\frac{1}{2} \sum_{i,j=1}^\infty \big( \varphi_{i+j} - \varphi_i - \varphi_j \big) \,  
(a_{i,j} \, f_i(t) \, f_j(t) -b_{i,j} \, f_{i+j}(t) )\ .
%\nonumber \\
%&&\hspace{3cm}
%- \frac{1}{2} \sum_{i,j=1}^\infty \big( \varphi_{i+j} - \varphi_i - \varphi_j \big) \,  b_{i,j} \, f_{i+j}(t)  . 
\label{eq:CF1_disc}
\end{eqnarray}
This equation can be recast as follows, in a form more suitable for describing the continuous-size analog:
\begin{eqnarray}
&&\hspace{-1cm}
\frac{d}{dt} \sum_{i=1}^\infty  \varphi_i \, f_i(t) =
\frac{1}{2} \sum_{i,j=1}^\infty \big( \varphi_{i+j} - \varphi_i - \varphi_j \big) \,  a_{i,j} \, f_i(t) \, f_j(t) \nonumber \\
&&\hspace{1.5cm}
- \frac{1}{2} \sum_{i=2}^\infty \Big( \sum_{j=1}^{i-1} \big( \varphi_i - \varphi_j - \varphi_{i-j} \big) \,  b_{j , i-j} \, \Big) f_i(t) \,. 
\label{eq:CF2_disc}
\end{eqnarray}
Taking $\varphi_i = i$, we obtain the formal conservation of total population
(corresponding to mass in physical systems): 
\begin{eqnarray}
&&\hspace{-1cm}
\frac{d}{dt} \sum_{i=1}^\infty i \, f_i(t) = 0\,. 
\label{eq:CF_mass_disc}
\end{eqnarray}
However, depending on the rates, it can happen that population decreases 
due to a flux to infinite size.  This phenomenon is known as gelation, associated with formation of infinite-size clusters).

The model written in \cite{Ma_etal_JTB11} essentially corresponds to a specific choice
of rate coefficients  which we take in the form
 \begin{eqnarray}
%&&\hspace{-1cm}
a_{i,j} = \alpha\,, \qquad b_{i,j} = \frac{\beta}{i+j-1}\,.
\label{eq:rates_niwa_disc} 
\end{eqnarray}
With these coefficients, the coagulation rate $\alpha$ is independent of cluster sizes.
Moreover, the overall fragmentation rate for the breakdown of clusters of size $i$ 
into clusters of any smaller size is
given by 
\[
\frac12\sum_{j=1}^{i-1}b_{j,i-j}=\frac12 \beta\,.
\]
This rate is constant, independent of $i$, and these clusters break into pairs 
with sizes $(1,i-1)$, $(2,i-2),\ldots(i-1,1)$ with equal probability.

As mentioned earlier, we have found that
a variant of this model  is far more accessible to analysis 
by  the transform methods which we will employ.  Namely, we can imagine that
clusters of size $i$ now split into pairs $(0,i)$, $(1,i-1),\ldots(i,0)$
with equal probability, and take the rate coefficients in the form
\begin{eqnarray}
%&&\hspace{-1cm}
a_{i,j} = \alpha\,, \qquad b_{i,j} = \frac{\beta}{i+j+1} \,.
\label{eq:rates_niwa_discD} 
\end{eqnarray}
We refer to the coagulation-fragmentation equations 
\qref{eq:CF3_disc}-\qref{eq:CF5_disc} with the coefficients 
in \qref{eq:rates_niwa_discD} as \textbf{Model D} (D for discrete size).
This model also arises in a natural way as a discrete approximation of 
Model C; see Section~14.
The overall effective fragmentation rate for clusters of size $i$
becomes $\frac12\beta\frac{i-1}{i+1}$,
because we do not actually have clusters of size zero.

In general, an equilibrium solution $f=(f_i)$ of equations 
\qref{eq:CF3_disc}-\qref{eq:CF5_disc} is in \textit{detailed balance}
if
\begin{equation}\label{eq:detailed}
a_{i,j}f_if_j = b_{i,j} f_{i+j}
\end{equation}
for all $i,j\in\N$. 
It is easy to see that for  neither choice of coefficients in 
\qref{eq:rates_niwa_disc} nor \qref{eq:rates_niwa_discD} 
do the equations  admit an equilibrium in detailed balance.
%\footnote{Discuss anything else? Strong vs weak fragmentation?
% The fact that Niwa's time scale $\hat t$ satisfies $m_0\,d\hat t = dt$?
% Perhaps not.}
 
 %
 %
 %
 \subsection{Continuous-size distributions}
 
 %\subsection{Continuous case}
%\label{subsec:CF_continuous}
To study systems with large populations and typical cluster sizes,
some simplicity is gained by passing to a continuum model in which 
clusters $(x)$ may assume a continuous range of sizes $x\in\rplus=(0,\infty)$.  
The corresponding reactions are  written schematically now as
\begin{eqnarray*}
&&\hspace{-1cm}
(x) + (y) \, \stackrel{a(x,y)}{\longrightarrow} \, (x+y) \quad \mbox{ (binary coagulation)}, \\
&&\hspace{-1cm}
(x) + (y) \, \stackrel{b(x,y)}{\longleftarrow} \, (x+y) \quad \mbox{ (binary fragmentation)}, 
\end{eqnarray*}
where $a(x,y)$ and $b(x,y)$ are the coagulation and  fragmentation rates, respectively. Again, both $a$ and $b$ are nonnegative and symmetric. 

The distribution of cluster sizes is now described in terms of a (cumulative) distribution function 
$F_t(x)$, which denotes the number density of clusters with size in $(0,x]$ at time $t$. 
According to probabilistic convention, we use the same notation $F_t$ to denote the
measure on $(0,\infty)$ with this distribution function; thus we write
\[
F_t(x) = \int_{(0,x]} F_t(dx)\,.
\]
The measure $F_t$ evolves according to the following size-continuous coagulation-fragmentation equation, which we write in weak form.
One requires that for any  suitable test function $\varphi (x)$,
%
%\begin{eqnarray}
%&&\hspace{-1cm}
%\frac{d}{dt} \int_{{\mathbb R}_+} \varphi(x) \, f(x,t) \, dx =
%\frac{1}{2} \int_{({\mathbb R}_+)^2} \big( \varphi (x+y) - \varphi(x) - \varphi(y) \big) \,  a(x,y) \, f(x,t) \, f(y,t)  \, dx \, dy \nonumber \\
%&&\hspace{2cm}
%- \frac{1}{2} \int_{({\mathbb R}_+)^2} \big( \varphi (x+y) - \varphi(x) - \varphi(y) \big) \,  b(x,y) f(x+y,t) \, dx \, dy . 
%\label{eq:CF1}
%\end{eqnarray}
%We note that this equation can be recast as follows:
%\begin{eqnarray} &&\hspace{-1cm}
\begin{equation}\label{eq:CF2} \begin{split}
&\frac{d}{dt} \int_{\rplus} \varphi(x) \, F_t(dx) = 
\frac{1}{2} \int_{{\mathbb R}_+^2} \big( \varphi (x+y) - \varphi(x) - \varphi(y) \big) a(x,y) 
\, F_t(dx) \, F_t(dy) 
%\nonumber \\&&\hspace{1.5cm}
\\ & \quad 
- \frac{1}{2} \int_{{\mathbb R}_+} \Big( \int_0^x \big( \varphi (x) - \varphi(y) - \varphi(x-y) \big) \,  b(y,x-y) \, dy \, \Big)  F_t(dx) . 
\end{split}\end{equation}
%\label{eq:CF2}
%\end{eqnarray}
%\begin{eqnarray}
%&&\hspace{-1cm}
%\frac{d}{dt} \int_{{\mathbb R}_+} \varphi(x) \, f(x,t) \, dx = 
%\frac{1}{2} \int_{({\mathbb R}_+)^2} \big( \varphi (x+y) - \varphi(x) - \varphi(y) \big) \,  a(x,y) \, f(x,t) \, f(y,t)  \, dx \, dy \nonumber \\
%&&\hspace{1.5cm}
%- \frac{1}{2} \int_{{\mathbb R}_+} f(x,t) \, \Big( \int_0^x \big( \varphi (x) - \varphi(y) - \varphi(x-y) \big) \,  b(y,x-y) \, dy \, \Big) \, dx . 
%\label{eq:CF2}
%\end{eqnarray}
%Although these equations are more easily analyzed in weak form, we also provide their strong form, which is as follows: 
%\begin{eqnarray}
%&&\hspace{-1.5cm}
%\frac{\partial f}{\partial t}(x,t) = Q_c(f)(x,t) + Q_f(f)(x,t) ,
%\label{eq:CF3}\\
%&&\hspace{-1.5cm}
%Q_c(f)(x,t) = \frac{1}{2} \int_0^x a(y,x-y) \, f(y,t) \, f(x-y,t) \, dy - \int_0^\infty a(x,y) \, f(x,t) \, f(y,t) \, dy , 
%\label{eq:CF4} \\
%&&\hspace{-1.5cm}
%Q_f(f)(x,t) = -\frac{1}{2} \int_0^x b(y,x-y) \, dy \, f(x,t) + \int_0^\infty b(x,y) \, f(x+y,t) \, dy\ . 
%\label{eq:CF5}
%\end{eqnarray}
%where $Q_c$ is the coagulation operator and $Q_f$ is the fragentation operator. 
Taking $\varphi(x) = x$, we obtain the formal conservation of total population: 
\begin{equation}
\frac{d}{dt} m_1(F_t)=0\,,
%\frac{d}{dt} \int_{{\mathbb R}_+} x \, F_t(dx) = 0. 
\label{eq:CF_mass}
\end{equation}
where in general, we denote the $k^{\rm th}$ moment of $F_t$ by
 \[
  m_k(F_t):=\int_{{\mathbb R}_+} x^k \, F_t(dx)\,. 
\]
However, again there might be loss of mass due to gelation, and now also mass flux to zero
may be possible (shattering, associated with the break-up of clusters to zero-size dust).
%\begin{eqnarray}
%&&\hspace{-1cm}
%\frac{d}{dt} \int_{{\mathbb R}_+} x \, f(x,t) \, dx \leq 0. 
%\label{eq:CF_mass_nc}
%\end{eqnarray}

The specific rate coefficients that we will study correspond to constant coagulation rates
and constant overall binary fragmentation rates with uniform distribution of fragments
are
%\item Model C (Continuous): 
\begin{eqnarray}
&&\hspace{-1cm}
a(x,y) = A\,, \qquad b(x,y) = \frac{B}{x+y}, \label{eq:rates_niwa_cont}
\end{eqnarray}
We refer to the coagulation-fragmentation equations 
\qref{eq:CF2} with these coefficients as 
\textbf{Model C} (C for continuous size).

For size distributions with density, written as $F_t(dx)=f(x,t)\,dx$,
Model C is written formally in strong form as follows:
\begin{eqnarray}
&&\hspace{-1.5cm}
\frac{\partial f}{\partial t}(x,t) = A\,Q_a(f)(x,t) + B\,Q_b(f)(x,t) ,
\label{eq:CF3_Niwa_11}\\
&&\hspace{-1.5cm}
Q_a(f)(x,t) = \frac12 \int_0^x \, f(y,t) \, f(x-y,t) \, dy -   \, f(x,t) \, \int_0^\infty f(y,t) \, dy , 
\label{eq:CF4_Niwa_11} \\
&&\hspace{-1.5cm}
Q_b(f)(x,t) = -\frac12 f(x,t)  +  \int_x^\infty \frac{f(y,t)}{y} \, dy . 
\label{eq:CF5_Niwa_11}
%\\
%&&\hspace{-1.5cm}
%m_1(f(t)) = \int_0^\infty x \, f(x,t) \, dx = 1, \quad \forall t \in [0,\infty). 
%\label{eq:norm_11}
\end{eqnarray}

To argue that there are no equilibrium solutions of Model C 
in detailed balance, we need a suitable definition of 
detailed balance for the equations \qref{eq:CF2} in weak form.
In section~\ref{sec:cts_balance} below we will propose such a definition, 
and verify that no finite measures on $(0,\infty)$ satisfy it. 

\subsection{Scaling relations for Models D and C}

By simple scalings, we can relate the solutions of Models D and C to solutions 
of the same models with conveniently chosen coefficients. 

Let us begin with Model D. Suppose $\hat f(t) = (\hat f_i(t))_{i\in\N}$ is
a solution of Model D for the particular coefficients $\alpha =\beta=2$. 
%\footnote{2 not 1! Fix the rest of this..}
Then a solution of Model D for general coefficients $\alpha$, $\beta>0$ is given by
\begin{equation}\label{eq:D_scale}
f_i(t) := \frac \beta\alpha \, \hat f_i(\beta t/2).
\end{equation}
For the purposes of analysis, then it suffices to treat the case $\alpha=\beta=2$,
as we assume below.

For Model C we have a similar expression. If $\hat F_t(x)$ is a weak solution
of Model C for particular coefficients $A=B=2$, then a
solution for general $A$, $B>0$ is given by 
\begin{equation}\label{eq:C_scale}
F_t(x) := \frac BA\, \hat F_{Bt/2}(x).
\end{equation}
Thus it suffices to deal with the case $A=B=2$ as below.

Importantly, Model C has an additional scaling invariance involving dilation of size.
If $\hat F_t(x)$ is any solution, for any $A$ and $B$, then 
\begin{equation}\label{scaleC:F}
F_t(x) := \hat F_t(Lx)
\end{equation}
is also a solution, for the same $A$ and $B$. 
Moreover, if we suppose that $\hat F_t(x)$ 
has first moment 
$\hat m_1 = \int_0^\infty x\,\hat F_t(dx)=1$, then
we can obtain
a solution $F_t(x)$ with arbitrary finite first moment 
\begin{equation}
m_1 = \int_0^\infty x\,F_t(dx)
\end{equation}
by taking $L=1/m_1$, so that
%is finite, we have $\hat m_1 = \int_0^\infty x\,\hat F_t(dx)=1$ provided
%$m_1=1/L$, so that
\begin{equation}
F_t(x) = \hat F_t(x/m_1)\,.
\end{equation}
For solutions with (fixed) total population, it therefore suffices to 
analyze the case that $m_1=1$.

%[Introduce general moments $m_k$ here?] 

%%%%%%%%%%%%%%%%%%%%%%%%%%%%%%%%%%%%%%%%%%%%%%%%%%%
%%%%%%%%%%%%%%%%%%%%%%%%%%%%%%%%%%%%%%%%%%%%%%%%%%%
\section{Bernstein functions and transforms}

Throughout this paper we make use of various results from the
theory of Bernstein functions, as laid out in the book \cite{Schilling_etal_Bernstein}.
We summarize here a number of key properties of Bernstein functions that we need
in the sequel. 

Recall that a function $g\colon(0,\infty)\to\R$ is \textit{completely monotone}
if it is infinitely differentiable and its derivatives satisfy $(-1)^ng\supn(x)\ge0$
for all real $x>0$ and integer $n\ge0$. By Bernstein's theorem, $g$ is completely
monotone if and only if it is the Laplace transform of some (Radon) measure on 
$[0,\infty)$. 

\begin{definition}
A function $U\colon (0,\infty)\to \R$ is a \textit{Bernstein function} if it is 
infinitely differentiable, nonnegative, and its derivative $U'$ is completely monotone.
%s satisfy
%\[
%(-1)^nU^{(n+1)}(s) \ge 0 \quad \mbox{for all integer $n\ge0$ and real $s\in(0,\infty)$}.
%\]
\end{definition}
The main representation theorem for these functions 
\cite[Thm. 3.2]{Schilling_etal_Bernstein}
associates to each Bernstein function $U$ a unique 
\textit{L\'evy triple} $(a_0,a_\infty,F)$ as follows.
(Below, the notation $a\wedge b$ means the minimum of $a$ and $b$.)
\begin{theorem}
 A function $U\colon(0,\infty)\to\R$ is a Bernstein function if and only if
it has the representation
\begin{equation}\label{def:Btransform}
U(s) = a_0s+a_\infty+\int_{(0,\infty)} (1-e^{-sx})\,F(dx)\,, \quad s\in(0,\infty),
\end{equation}
where $a_0$, $a_\infty\ge0$ and $F$ is a measure satisfying
\begin{equation}
\int_{(0,\infty)} (x\wedge1)F(dx)<\infty. 
\end{equation}
In particular, the triple $(a_0,a_\infty,F)$ uniquely determines $U$ and vice versa.
\end{theorem}

We point out that $U$ determines $a_0$ and $a_\infty$ via the relations
\begin{equation}\label{eq:a0ainfty}
a_0 = \lim_{s\to\infty} \frac{U(s)}s\,, \qquad
a_\infty = U(0^+)=\lim_{s\to0}U(s)\, .
\end{equation}

\begin{definition} Whenever \qref{def:Btransform} holds, we call $U$
the \textbf{Bernstein transform} of the L\'evy triple
$(a_0,a_\infty,F)$.  If $a_0=a_\infty=0$, we call $U$ the Bernstein
transform of the L\'evy measure $F$, and write $U=\breve F$.
\end{definition}

%
%A function $U$ is Bernstein if and only if it is nonnegative and its derivative $U'$ 
%is completely monotone. By Bernstein's theorem, this means $U'$ 
%is a Laplace transform.  Thus, 
Many basic properties of Bernstein functions follow from
Laplace transform theory.   Yet Bernstein functions have beautiful
and distinctive properties that are worth delineating separately. 
The second statement in the following proposition is proved in Lemma~2.3
of \cite{ILP2015}.
\begin{proposition}\label{prop:compose}
The composition of any two Bernstein functions is Bernstein. 
Moreover, if $V\colon(0,\infty)\to(0,\infty)$ is bijective and $V'$ is Bernstein,
then the inverse function $V\inv$ is Bernstein.
\end{proposition}

{\sl Topologies.} Any pointwise limit of a sequence of Bernstein functions is Bernstein.
The topology of this pointwise convergence corresponds to a notion of weak convergence
related to the associated L\'evy triples in a way that is not fully characterized
in \cite{Schilling_etal_Bernstein}, but may be described as follows. 
Let $\calM_+[0,\infty]$ denote the space of nonnegative finite (Radon) measures on
the compactified half-line $[0,\infty]$. 

To each L\'evy triple $(a_0,a_\infty,F)$ we associate a finite measure 
$\kappa\in\calM_+[0,\infty]$ by appending atoms of magnitude
$a_0$ and $a_\infty$ respectively at $0$ and $\infty$ to the finite measure $(x\wedge1)F(dx)$, writing
\begin{equation}\label{def:kappa}
\kappa(dx) = a_0\delta_0 + a_\infty\delta_\infty + (x\wedge1)F(dx).
\end{equation} 
The correspondence $ (a_0,a_\infty,F)  \mapsto\kappa$ is bijective,
as is the correspondence $(a_0,a_\infty,F)\to U$.
We say that a sequence of finite measures $\kappa_n$ \textit{converges weakly}
on $[0,\infty]$ to $\kappa$ and write $\kappa_n\wkto\kappa$ 
if  for each continuous $g\colon[0,\infty]\to\R$,
\[
%\ip{\kappa_n,g} \to \ip{\kappa,g}:=\int_{[0,\infty]} g(x)\,\kappa(dx) 
\int_{[0,\infty]} g(x)\,\kappa_n(dx) \to \int_{[0,\infty]} g(x)\,\kappa(dx) 
\quad\mbox{as $n\to\infty$}.
\]
\begin{theorem}\label{th:Btopology}
Let $(a_0\supn,a_\infty\supn,F\supn)$, be a sequence of L\'evy triples,
with corresponding Bernstein transforms $U_n$ and $\kappa$-measures $\kappa_n$
on $[0,\infty]$. Then the following are equivalent:
\begin{itemize}
\item[(i)]  $U(s):=\lim_{n\to\infty}U_n(s)$ exists for each $s\in(0,\infty)$. 
\item[(ii)] $\kappa_n \wkto \kappa$ as $n\to\infty$, where $\kappa$ is a finite
measure on $[0,\infty]$. 
\end{itemize}
If either (i) or (ii) hold, then the respective limit $U$, $\kappa$ is
associated with a unique L\'evy triple $(a_0,a_\infty,F)$. 
\end{theorem}
This result is essentially a restatement of Theorem~3.1 in \cite{MP2008},
where different terminology is used. 
A simpler, direct proof will appear in \cite{ILPxx}, however.

For finite measures on $[0,\infty)$, the topology of pointwise convergence 
of Bernstein functions is related to more usual notions of weak convergence as follows.
Let $\calM_+[0,\infty)$ be the space of nonnegative finite (Radon) measures
on $[0,\infty)$.
Given $F$,  $F_n\in \calM_+[0,\infty)$ for $n\in\N$, we say $F_n$
converges to $F$ \textit{vaguely} on $[0,\infty)$ and write $F_n\vto F$
provided  
\begin{equation}\label{lim:vlim}
\int_{[0,\infty)} g(x)\, F_n(dx) \to \int_{[0,\infty)} g(x)\, F(dx) 
\end{equation}
for all functions $g\in C_0[0,\infty)$, the space of continuous functions
on $[0,\infty)$ with limit zero at infinity. 
We say $F_n$ converges to $F$ \textit{narrowly} and write $F_n\nto F$ 
if \qref{lim:vlim} holds for all $g\in C_b[0,\infty)$, the space of
bounded continuous functions on $[0,\infty)$. 
As is well-known  \cite[p.~264, Thm.~13.35]{Klenke}, $F_n\nto F$ if and only if 
both $F_n\vto F$ and $m_0(F_n)\to m_0(F)$.

Denote the Laplace transform of any $F\in\calM_+[0,\infty)$ by 
\begin{equation}
\calL F(s) = \int_{[0,\infty)} e^{-sx} F(dx)\,.
\end{equation}
By the usual Laplace continuity theorem,  
$F_n\vto F$ if and only if $\calL F_n(s)\to\calL F(s)$ for all $s\in(0,\infty)$. 
Regarding narrow convergence, we will make use of the following result
for  the space $\calM_+(0,\infty)$ consiting of finite nonnegative measures 
on $(0,\infty)$.  We regard elements of this space as elements of 
$\calM_+[0,\infty)$ having no atom at $0$.

\begin{proposition} \label{prop:narrow}
Assume $F$, $F_n\in\calM_+(0,\infty)$ for $n\in\N$. Then the following are equivalent
as $n\to\infty$.
\item[(i)] $F_n$ converges narrowly to $F$, i.e., $F_n\nto F$.
\item[(ii)]  The Laplace transforms $\calL F_n(s)\to\calL F(s)$, for each $s\in[0,\infty)$.
\item[(iii)] The Bernstein transforms $\breve F_n(s) \to \breve F(s)$, for each $s\in[0,\infty]$.
\item[(iv)] The Bernstein transforms $\breve F_n(s) \to \breve F(s)$, uniformly for $s\in(0,\infty)$. 
\end{proposition}

\begin{proof}
The equivalence of (i) and (ii) follows from the discussion above. 
Because 
\[
\breve F_n(s) = \int_{(0,\infty)} (1-e^{-sx})F_n(dx)\,,
\qquad \breve F_n(\infty) = m_0(F_n)=\calL F_n(0)\,,
\]
the equivalence of (ii) and (iii) is immediate. Now
the equivalence of (iii) and (iv) follows from 
an extension of Dini's theorem due to P\'olya
\cite[Part II, Problem 127]{PolyaSzego},
%the following elementary lemma, 
because $\breve F_n$ and $\breve F$ are continuous and monotone
on the compact interval $[0,\infty]$.
\end{proof}

{\sl Complete monotonicity.}
Our analysis of the equilibria of Model C relies on a striking result from 
the theory of so-called  \textit{complete} Bernstein functions, 
as developed in \cite[Chap. 6]{Schilling_etal_Bernstein}:
\begin{theorem} \label{thm:CBF}
The following are equivalent.
\begin{itemize}
\item[(i)] The L\'evy measure $F$ in \qref{def:Btransform}
has a completely monotone density $g$, so that
\begin{equation}\label{eq:CBF}
U(s) = a_0s+a_\infty+\int_{(0,\infty)} (1-e^{-sx}) g(x)\,dx\,, \quad s\in(0,\infty).
\end{equation}
\item[(ii)] $U$ is a Bernstein function that admits a holomorphic extension to the cut plane
$\C\setminus(-\infty,0]$ satisfying $(\im s)\im U(s) \ge0$. 
\end{itemize}
\end{theorem}
In complex function theory,
a function holomorphic on the upper half of the complex plane that leaves
it invariant is called a \textit{Pick function} or \textit{Nevalinna-Pick function}.
%[Nevalinna, Herglotz?]. 
Condition (ii) of the theorem above says simply that $U$ is a Pick function
analytic and nonnegative on $(0,\infty)$. 
Such functions are called \textit{complete Bernstein functions} in 
\cite{Schilling_etal_Bernstein}.

For our analysis of Model D,  we will make use of an analogous criterion 
recently developed regarding sequences $(c_i)_{i\ge0}$. 
Such a sequence is called completely monotone
if all its backward differences are nonnegative:
\begin{equation}
(I-S)^k c_j \ge 0 \quad\mbox{for all $j,k\ge0$,}
\end{equation}
where $S$ denotes the backshift operator, $Sc_j=c_{j+1}$. 
The following result follows immediately from Corollary 1 in \cite{LP2016}.
%[Fix to quote Corollary 1 instead.]
\begin{theorem}\label{thm:Pick}
Let $c=(c_j)_{j\ge0}$ be a real sequence with generating function
\begin{equation}\label{d:Gen}
G(z) = \sum_{j=0}^\infty c_j z^j \ .
\end{equation}
%and upshifted generating function 
%\begin{equation}\label{d:FB}
%G_1(z) = z G(z) = \sum_{j=0}^\infty  c_j z^{j+1}\,.
%\end{equation}
Let $\lambda>0$.  Then the following are equivalent.
\begin{itemize}
\item[(i)] The sequence $(c_j\lambda^j)$ is completely monotone.
\item[(ii)] $G$ is a Pick function that is analytic and nonnegative on $(-\infty,\lambda)$.
%\item[(iii)] $G_1$ is a Pick function that is analytic on $(-\infty,1)$, with $G_1(0)=0$.
\end{itemize}
\end{theorem}

\vfil\pagebreak

\part{Analysis of Model C}
\myrunningheads

\section{Equations for the continuous-size model}

In this part of the paper we analyze the dynamics of Model C,
requiring $A=B=2$ as discussed above.
 Weak solutions of the model are then required to satisfy
 (the time-integrated version of) the equation
\begin{equation}\label{eq:modelC} 
\begin{split}
  \frac{d}{dt} \int_{{\mathbb R}_+} 
\varphi(x) \, F_t(dx) &=  \int_{{\mathbb R}_+^2} 
\big( \varphi (x+y) - \varphi(x) - \varphi(y) \big) 
 F_t(dx) \, F_t(dy) 
%\nonumber \\&&\hspace{1.5cm}
\\  &\qquad 
-  \int_{{\mathbb R}_+} 
%\Big(  
\frac1x \int_0^x \big( \varphi (x) - 2\varphi(y) \big) \, dy 
%\Big)
\,F_t(dx) ,
\end{split}\end{equation}
for all continuous functions $\varphi$ on $[0,\infty]$. 

In successive sections to follow, we study the existence and uniqueness of weak solutions, 
characterize the equilibrium solutions, and study the long-time behavior of solutions
for both finite and infinite total population size. 

\textit{Bernstein transform.} 
The results will be derived through study of the equation corresponding
to \qref{eq:modelC} for the Bernstein transform 
\begin{equation}
\uu(s,t)=\breve F_t(s) = \int_{\rplus} (1-e^{-sx})\,F_t(dx)\,.
\end{equation} 
We obtain the evolution equation for $\breve F_t$ by taking 
$\varphi_s(x)=1-e^{-sx}=\varphi_1(sx)$ as test function 
 in \qref{eq:modelC} 
and using the identities
\begin{equation}
-\varphi_s(x)\varphi_s(y)= \varphi_s(x+y)-\varphi_s(x)-\varphi_s(y)\,,
\end{equation}
\begin{eqnarray*}
\frac{1}{x} \int_0^x \varphi_1(sy) \, dy 
&=& 1 - \frac{1}{sx} \varphi_1(sx) =
\frac{1}{s} \int_0^s \varphi_1(r x) \, dr \,.
\end{eqnarray*}
Then for any weak solution $F_t$, the Bernstein transform must satisfy
\begin{equation}
\boxed{
\frac{\D\uu}{\D t}(s,t) = -\uu^2 -\uu + \frac2s \int_0^s \uu(r,t)\,dr\,.
}
\label{eqC:Bernstein}
\end{equation}
This equation is not as simple as in the case of pure coagulation studied
in \cite{MP2004}, without linear terms, 
but it turns out to be quite amenable to analysis. 

Eq.~\qref{eqC:Bernstein} has a dilational symmetry that 
correspondings to the symmetry that appears in \qref{scaleC:F}.
Namely, if $\hat\uu(s,t)$ is any solution of \qref{eqC:Bernstein}, then
\begin{equation}\label{scaleC:U}
U(s,t) = \hat U(s/L,t)
\end{equation}
is another solution, for any constant $L>0$. 

Furthermore, the zeroth moment of any solution satisfies a logistic equation:
By taking $\varphi\equiv1$ in \qref{eq:modelC}, we find that 
\begin{equation}\label{def:m0}
m_0(F_t)= \int_{(0,\infty)}F_t(dx)
\end{equation}
satisfies
\begin{equation}\label{eq:m0evolve}
\frac{d}{dt} m_0(F_t) = -m_0(F_t)^2 + m_0(F_t)\,.
\end{equation}

%%%%%%%%%%%%%%%%%%%%%%%%%%%%%%%%%%%%%%%%%%%%%%%%%%%%
%%%%%%%%%%%%%%%%%%%%%%%%%%%%%%%%%%%%%%%%%%%%%%%%%%%%
\section{Equilibrium profiles for Model C} \label{sec:equilibriumC}

%According to the theory in the previous section, see \qref{eq:m0evolve}, 
We begin the analysis of Model C by characterizing its equilibrium solutions.
Due to \qref{eq:m0evolve}, 
any non-zero steady-state solution $F_{\rm eq}(dx)$ of Model C must have 
 zeroth moment 
 \begin{equation}
 m_0(F_{\rm eq})=1\ .
 \end{equation}
In this section, we prove the following theorem. 

\begin{theorem}\label{thm:cont_CM}
There is universal smooth profile $f_\star\colon\rplus\to\rplus$,
such that 
every non-zero steady-state solution $F_{\rm eq}$ of Model C has 
a smooth density, taking the form $F_{\rm eq}(dx)=\feq(x)\,dx$, where
\begin{equation}\label{eqC:feqscale}
\feq(x) = \frac{1}{\mu} f_\star\left(\frac x{\mu}\right) \,,
\end{equation}
with first moment
\[
\mu = m_1(\Feq) = \int_\rplus x\,\feq(x)\,dx\in(0,\infty)\,.
\]
The universal profile $f_\star$ can be written as
\begin{equation}
    f_\star(x) = \gamma_\star(x) \, e^{- \frac{4}{27} x}, 
\label{eq:cont_CM_2}
\end{equation}
where $\gamma_\star$ is a completely monotone function. Furthermore, $\gamma_\star$ has the following asymptotic behavior:
\begin{eqnarray}
    \gamma_\star(x) &\sim& \phantom{\frac98}\frac{x^{-2/3}}{ \Gamma({1}/{3})} \, , \qquad \mbox{when} \quad x \to 0, 
\label{eq:cont_CM_3} \\
\gamma_\star(x) &\sim& \frac98 \frac{x^{-3/2}}{ \Gamma(1/2)} \, , \qquad \mbox{when} \quad x \to \infty,  
\label{eq:cont_CM_4} 
\end{eqnarray}
and the profile satisfies 
\begin{equation}
1 = \int_0^\infty f_\star(x)\,dx 
= \int_0^\infty x f_\star(x)\,dx
= \frac16 \int_0^\infty x^2 f_\star(x)\,dx \,.
%m_1(f_\star) = \frac16 m_2(f_\star).
\label{eq:cont_CM_4_1}
\end{equation}
\end{theorem}

\begin{remark}
The profile $f_\star$ is related to the profile $\PhiNew$ 
that appeared in \qref{eq:feqPhi} as follows.
Given any equilibrium density $\feq(x)=f_\star(x/m_1)/m_1$, with $m_2=m_2(\feq)$
we compute $\xav =  {m_2}/{m_1} = 6m_1$, and therefore 
\[
 \feq(x) = \frac 6{\xav} f_\star\left(\frac {6x}{\xav}\right)\,. \]
Thus we recover \qref{eq:feqPhi} and \qref{Phi-ours} with
\begin{equation}\label{Phi-fstar}
    \PhiNew(x) = 6 f_\star(6x)\,, \qquad g(x) = 6\gamma_\star(6x)\,.
\end{equation}
\end{remark}
%This follows from the moment relations \qref{eq:cont_CM_4_1}, and yields
%the relations \qref{Phi-ours}-\qref{eq:profile4} stated in the introduction. 
\begin{remark} We mention that the same profiles as in \qref{eqC:feqscale} determine
equilibrium solutions for coagulation-fragmentation equations with
coagulation and fragmentation rates having the power-law form
\begin{equation}
a(x,y) = x^\alpha y^\alpha\,, \qquad
b(x,y) = (x+y)^{\alpha-1}\,.
\end{equation}
Namely, for these rates, equilibria exist with the form
\begin{equation}\label{eqC:feqscalepow}
\feq(x) = \frac{x^{-\alpha}}{\mu} f_\star\left(\frac x{\mu}\right) \,.
\end{equation}
\end{remark}

\subsection{The Bernstein transform at equilibrium}
The proof of Theorem \ref{thm:cont_CM} starts by determining the possible 
steady solutions of equation \qref{eqC:Bernstein} for the Bernstein transform. 
A time-independent solution $U = \breve F_{\rm eq}$ of \qref{eqC:Bernstein} must satisfy
\begin{eqnarray}
&&\hspace{-1cm}
{U(s)^2 +U(s) = \frac{2}{s}  \int_0^s U(r) \, d r \,.} 
\label{eq:Uint1}
\end{eqnarray}
Differentiating with respect to $s$ and eliminating the integral yields
%or again
%\begin{eqnarray*}
%&&\hspace{-1cm}
%s(U^2(s) + U(s)) = 2  \int_0^s U(\sigma) \, d \sigma , 
%\end{eqnarray*}
%By differentiating this equation, we get
%$$ (U^2(s) + U(s)) + s(U^2 + U)'(s) = 2U(s), $$
%or equivalently
\begin{equation}
\frac{1}{s}= \frac{(2 U + 1) U'}{U(1-U)} =\frac{U'}{U} + \frac{3U'}{1-U} \,. 
\end{equation}
%or again
%$$  = \frac{1}{s}. $$
Because we must have $0<U(s)<1= m_0(F)$ for all $s$,
every relevant solution of this differential equation must satisfy
$U(s)=U_\star(Cs)$ for some positive constant $C$, where
\begin{equation}\label{eq:Usolution}
    \boxed{\frac{U_\star(s)} {(\,1-U_\star(s))^3} = s\,. }
\end{equation}
Because $U(0)=0$ it follows necessarily that 
\begin{equation}
m_1(F_{\rm eq}) = \int_{\rplus} x\,F_{\rm eq}(dx)
=  U'(0^+)  = C  \,.
\end{equation}
Thus, in order to prove the theorem, it is necessary to show that  the 
solution $U_\star$ of \qref{eq:Usolution} is a Bernstein function taking the form
\begin{equation}
U_\star(s) = \breve f_\star(s) = \int_0^\infty (1-e^{-sx}) \gamma_\star(x)e^{-\frac4{27}x}\,dx
\end{equation}
where $\gamma_\star$ has the properties as stated. 
(We note that $-\frac4{27}$ is the minimum value of the function $u\mapsto u/(1-u)^3$, 
at $u=-\frac12$.)

%But, we know from above that $U'(0) = 1$. 
%Therefore $C=1$ and we recover (\ref{eq:cont_CM_14}). By proceeding backwards, this shows that (\ref{eq:cont_CM_14}) implies (\ref{eq:cont_CM_15}). 

\smallskip
\noindent{\bf Remark.} An explicit expression for the solution of \qref{eq:Usolution}
is given by 
\begin{equation} 
U_\star(s)= 
\frac{1}{\sqrt{s}} \left( \left( \frac{\sqrt{s + s_0} + \sqrt{s}}{2} \right)^{1/3} - \left( \frac{\sqrt{s + s_0} - \sqrt{s}}{2}\right)^{1/3} \right)^3 , 
\label{eq:cont_Niwa_equi_3} 
\end{equation}
with $ s_0 = \frac{4}{27}.$
But we have not managed to derive any significant consequences from this. 
\smallskip

\subsection{Relation to Fuss-Catalan numbers}

The Bernstein transform $U_\star(s)=\breve f_\star(s)$ of the universal equilibrium
profile for Model C turns out to be very closely related to the generating function of the
 {\em Fuss-Catalan numbers} with single parameter $p$,
given by
\begin{equation}
 B_p(z) = \sum_{n=0}^\infty \binom{pn+1}{n}\frac{z^n}{pn+1}\ .
\end{equation}
This function is well-known \cite{Graham,Mlotkowski} to satisfy
\begin{equation}\label{eq:Bp}
B_p(z) = 1 + z B_p(z)^p \ .
\end{equation}
Consequently, as there is a unique real solution to $X=1-sX^3$ for $s>0$,  we find from \qref{eq:Usolution} that
\begin{equation}\label{eqC:B3U}
%{ U_\star(s) = 1 - B_3(-s) \ ,}
B_3(-s) = 1- U_\star(s) = \int_0^\infty e^{-sx} f_\star(x)\,dx\,,
\end{equation}
and therefore
\begin{equation}
U_\star(s) = -\sum_{n=1}^\infty \binom{3n+1}{n} \frac{(-s)^n }{3n+1}\ .
\end{equation}
The function $B_3$ has the following properties, established in
\cite[Lemma 3]{LP2016}.
\begin{lemma}\label{L:B3}
The function $B_3$ is a Pick function analytic on $\C\setminus[\frac4{27},\infty)$
and nonnegative on the real interval $(-\infty,\frac4{27})$.
\end{lemma}

\subsection{Proof of Theorem \ref{thm:cont_CM}.} 

1. We infer from Lemma~\ref{L:B3} that $U_\star$ is a Pick function analytic on $(-\frac4{27},\infty)$.
Because $B_3(\frac4{27})=\frac32$ due to \qref{eq:Bp} , we have
\begin{equation}
U_\star\left(-\frac4{27}\right) = -\frac12\,.
\end{equation}
We make the change of variables 
\begin{equation} \label{eq:UtoV}
V = U_\star + \frac{1}{2}\,, \qquad z = s + \frac{4}{27}\,.
\end{equation}
Then $V$ is a Pick function analytic and nonnegative on $(0,\infty)$. 
Invoking Theorem~\ref{thm:CBF}, we deduce that $V$ is a Bernstein function
whose L\'evy measure has a completely monotone density $\gamma_\star$. 
Note that $V(0^+)=0$, and from \qref{eq:Usolution} we have 
$U_\star(s)\to1$ as $s\to\infty$, hence
\[
\lim_{z\to\infty} \frac{V(z)}z = \lim_{s\to\infty}\frac{U_\star(s)}s = (1-U_\star(\infty))^3 =0.
\]
Therefore by Theorem~\ref{thm:CBF} and  \qref{eq:a0ainfty}
it follows that
\begin{equation}\label{eq:Vgamma}
    V(z) = \int_0^\infty (1-e^{-zx})\gamma_\star(x)\,dx \,,
\end{equation}

2. We will analyze the asymptotic behavior of $\gamma_\star(x)$ using Tauberian arguments,
starting with the range $x\sim0$. Because $1-U_\star \sim s^{-1/3}$
as $s\to\infty$ due to \qref{eq:Usolution}, we can write
\begin{equation}
    \frac{3}{2} - V(z)  = \int_0^\infty e^{-zx}\gamma_\star(x)\,dx  \sim z^{-1/3} \qquad \mbox{as} \quad z \to \infty. 
\label{eq:cont_CMV2}
\end{equation}
By a result that follows from Karamata's Tauberian theorem \cite[Thm. XIII.5.4]{Feller}, it follows
%\[
%\frac1{y^{1/3}}\int_0^y \gamma_\star(x)\,dx \to \frac{1}{\Gamma(4/3)}
%\qquad\mbox{as $y\to0$}.
%\]
%By L'H\^opital's rule [??seems fishy] it follows
\[
    \gamma_\star(x) \sim \frac{ x^{-2/3}}{\Gamma(1/3)}
\qquad\mbox{as $x\to0$}.
\]

3. Next we deal with limiting behavior as $x\to\infty$. Transformation of \qref{eq:Usolution}
yields
\begin{eqnarray}
G(V(z)) = z, \qquad G(V) = \frac{16 \, V^2 \, (9 - 2V)}{27 \, (3-2V)^3}. 
\label{eq:cont_CM_11} 
\end{eqnarray}
We notice that 
\begin{eqnarray}
G(0) = G'(0) = 0 \, < \,  G''(0) = 2 \, \left( \frac{4}{9} \right)^2, 
\label{eq:cont_CM_12} 
\end{eqnarray}
and that $G$ is monotone increasing from $[0,\frac32)$ onto $[0,+\infty)$. Denote by $G^{-1}(z)$ the inverse function of $G$, which is monotone increasing from $[0,+\infty)$ onto $[0,\frac32)$. Eq. (\ref{eq:cont_CM_11}) is equivalent to saying  $V(z) = G^{-1}(z)$.

Thanks to (\ref{eq:Usolution}), we can use Taylor's theorem to write 
$$ G(V) %= \frac{G''(0)}{2} V^2 \, (1+h(V))^2
= \big(\frac 49V\big)^2(1+h(V))^2 , \qquad \mbox{as} \quad V \to 0 , $$
where $h(V)$ is analytic in the neighborhood of $V=0$ and is such that $h(0) = 0$. Now, introducing $\zeta = \sqrt{z}$ where the complex square root is taken with branch cut 
along the interval $(- \infty,0)$, Eq. (\ref{eq:cont_CM_11}) is written
$$ %\sqrt{\frac{G''(0)}{2}} 
\frac49\tilde{V} \, (1+h(\tilde{V})) = \zeta, $$
where $\tilde{V}(\zeta) = V(z)$. This equation shows that $\tilde V$ is an analytic function of $\zeta$ in the neighborhood of $\zeta = 0$, such that  
$$ \tilde{V} (\zeta) = %\sqrt{\frac{2}{G''(0)}} 
\frac94 \zeta + O(\zeta^2), \qquad
 \tilde{V}' (\zeta) = %\sqrt{\frac{2}{G''(0)}} +
\frac94 + O(\zeta).
$$ %\mbox{h.o.t.}.$$
%In particular, we deduce that $$  % \mbox{h.o.t.}.$$
Going back to $V(z)$, %and using (\ref{eq:cont_CM_12}), 
we find that as $z\to0$,
$V(z)\sim\frac94z^{1/2}$ and 
%\begin{equation}
% V(z) \sim  \frac{9}{4}\, z^{1/2}, \qquad V'(z) \sim  \frac{9}{8}\, z^{-1/2}, \qquad \mbox{as} %\quad z \to 0. 
%\label{eq:cont_CMV1}
%\end{equation}
%Finally, from (\ref{eq:cont_CM_3}) we have
\begin{equation}\label{eqC:Vz0}
    \int_0^\infty e^{-zx}x\gamma_\star(x)\,dx  =
V'(z)\sim \frac98 z^{-1/2} \,.
%\qquad\mbox{as $z\to0$}.
\end{equation}
Although the Tauberian theorem cited previously does not
apply directly (because we do not know $x\mapsto x\gamma_\star(x)$ is monotone), 
the selection argument used in the proof of Theorem~XIII.5.4 from~\cite{Feller}
works without change. 
It follows
%\[
%\frac1{y^{1/2}}\int_0^y x\gamma_\star(x)\,dxy \to \frac{9}{8\Gamma(3/2)}
%\qquad\mbox{as $x\to\infty$}.
%\]
%By L'H\^opital's rule [??check] we may then deduce that
\[ x\gamma_\star(x) \sim \frac98 \frac{x^{-1/2}}{ \Gamma(1/2)}
\qquad\mbox{as $x\to\infty$}.
\]
This proves \qref{eq:cont_CM_4} and completes the characterization of $\gamma_\star$.

4. Finally, by using \qref{eq:Vgamma} we may express $U_\star$ in the form
\begin{equation} 
    U_\star(s) = \int_0^\infty \Big( 1-e^{-\big( s+\frac{4}{27} \big) x} \Big) \, \gamma_\star(x) \, dx - \frac{1}{2}. 
\label{eq:cont_CM_13}
\end{equation}
Note that thanks to \qref{eq:Vgamma} we have
$$ \int_0^\infty (1- e^{-\frac{4}{27} x} ) \, \gamma_\star(x) \, dx = V(\frac{4}{27}) 
= U_\star(0) + \frac{1}{2} = \frac{1}{2}. $$
Therefore, we can recast (\ref{eq:cont_CM_13}) into:
\begin{eqnarray} 
    U_\star(s) &=& \int_0^\infty \Big( 1-e^{-\big( s+\frac{4}{27} \big) x} \Big) \, \gamma_\star(x) \, dx - \int_0^\infty (1- e^{-\frac{4}{27} x} ) \, \gamma_\star(x) \, dx \nonumber\\
         &=&  \int_0^\infty \big( 1-e^{-s x} \big) \, e^{-\frac{4}{27} x} \, \gamma_\star(x) \, dx.
\label{eq:cont_CM_14}
\end{eqnarray}

By the uniqueness property of Bernstein transforms, this proves that 
$F_{\rm eq}(dx)=f_\star(x)\,dx$ where $f_\star$ is given by
\qref{eq:cont_CM_2} as claimed.
It remains only to discuss the moment 
relations in \qref{eq:cont_CM_4_1}.  We already know 
\[
m_0(f_\star)=U_\star(\infty)=1, \qquad
m_1(f_\star)=U_\star'(0^+)=1.
\]
By multiplying \qref{eq:Usolution} by $(1-U_\star)^3$ and differentiating twice, one finds \[m_2(f_\star)=-U_\star''(0^+)=6. \]
This finishes the proof of Theorem~\ref{thm:cont_CM}.
\endproof

%%%%%%%%%%%%%%%%%%%%
\subsection{Series expansion for the equilibrium profile}
\label{sec:seriesC}

Start with the equation \qref{eq:Usolution}
and change variables via 
\begin{equation}
s=z^{-3},\qquad W(z) =1-U_\star(s) = \int_0^\infty e^{-sx}f_\star(x)\,dx\,,
\end{equation}  %$, $s=z^3$ 
to obtain 
\[
\frac{W}{\phi(W)} = z \,, \qquad \phi(W) = (1-W)^{1/3}\,.
\]
By the Lagrange inversion formula (see \cite[p.~128-133]{WhittakerWatson} and \cite{Henrici}), we find that
\begin{equation}
W = \sum_{n=1}^\infty a_n z^n = \sum_{n=1}^\infty a_n s^{-n/3}, 
\label{W-expand}
\end{equation}
where $na_n$
is the coefficient of $w^{n-1}$ in the (binomial) series expansion of $\phi(w)^n$. 
Thus we find
\begin{equation}
a_n = \frac{ (-1)^{n-1}}n \binom{n/3}{n-1} 
%= \frac{(-1)^{n-1}}{n!} \frac{\Gamma(1+n/3)}{\Gamma(2-2n/3)}
= \frac{(-1)^{n-1}}{3(n-1)!} \frac{\Gamma(n/3)}{\Gamma(2-2n/3)}. 
\end{equation}
Term-by-term Laplace inversion using 
$\int_0^\infty e^{-sx}x^{p-1}\,dx = \Gamma(p) s^{-p}$ 
with $p=n/3$ yields
\begin{equation}
%= \frac1{3x} \sum_{n=1}^\infty \frac{ (-1)^{n-1}}{(n-1)!}
%\frac {x^{n/3}}{\Gamma(2-2n/3)} 
\boxed{ f_\star(x) 
= \frac {x^{-2/3}}{3} \sum_{n=0}^\infty 
\frac{(-1)^n}{\Gamma(\frac43-\frac23 n)} \frac{x^{n/3}}{n!} . 
}
\end{equation}
Every third term of this series vanishes because of the poles of
the Gamma function at negative integers.
%Due to Euler's reflection formula for the Gamma function, we have
%\[
%\frac{1}{\Gamma(\frac43-\frac23 n)} = \frac1\pi 
%\sin \frac{\pi(2n-1)}3 \Gamma( \frac{2n-1}3 ).
%\]
Replacing $n$ by $3k$, $3k+1$, $3k+2$ respectively yields zero 
in the last case, and the series reduces to 
\begin{equation}
f_\star(x) = \frac {x^{-2/3}}{3} 
\sum_{k=0}^\infty \left(
 \frac{(-1)^k}{\Gamma(\frac43-2k)}\frac{x^k}{(3k)!}
+ \frac{(-1)^{k+1}}{\Gamma(\frac23-2k)} \frac{x^{k+1/3}}{ (3k+1)!}
\right).
%\frac{\sqrt 3 x^{-2/3}}{6\pi}
%\sum_{k=0}^\infty \frac{(-1)^{k+1} x^k}{(3k)!}
%\left(
%\Gamma(2k-\frac13) + \frac{x^{1/3}}{3k+1} \Gamma(2k+\frac13)
%\right)
\end{equation}

\setcounter{equation}{0}
%%%%%%%%%%%%%%%%%%%%%%%%%%%%%%%%%%%%%%%%%%%%%%%%%%%%
%%%%%%%%%%%%%%%%%%%%%%%%%%%%%%%%%%%%%%%%%%%%%%%%%%%%
\section{Global existence and mass conservation} \label{sec:existenceC}
%\section{Analysis of the discrete model D}\label{sec:disc_model}

Next, we deal with the initial-value problem for Model C.
Although this can be studied using rather standard techniques from kinetic theory, 
we find it convenient to study this problem using Bernstein function theory.

We require solutions take values in the space $\calM_+(0,\infty)$, which we recall to be the 
space of non-negative finite measures on $(0,\infty)$.
The main result of this section is the following.

\begin{theorem}\label{thm:EU-modelC}
Given any $\Finit\in \calM_+(0,\infty)$, there is a unique narrowly continuous
map $t\mapsto F_t\in \calM_+(0,\infty)$, $t\in(0,\infty)$ that satisfies \qref{eq:modelC}.
% $ \int_{0}^{\infty}(x\wedge1)dF_{0}(x)<\infty$. 
%Then there is a global measure solution $dF_{t}(x)$ weakly continuous in time that 
%satisfies 
%\[
%  \int_{0}^{\infty}(x\wedge1)dF_{t}(x) < \infty, \quad \forall t > 0
%\]
Furthermore, 
\begin{itemize}
\item[(i)] If the first moment $m_{1}(\Finit)<\infty$, then
\begin{equation}\label{eq:m1const}
m_1(F_t) = m_1(\Finit)
 \end{equation}
 for all $t\in[0,\infty)$, and $t\mapsto x\,F_t(dx)\in\calM_+(0,\infty)$ is narrowly continuous.
\item[(ii)] If $\Finit(x)= \finit(x)\,dx$ where $\finit$ is a completely monotone function, then for any $t \ge 0$, $F_{t}(dx)= f_{t}(x)\,dx$
where  $f_{t}$ is also a completely monotone function.
\end{itemize}
\end{theorem}  
\begin{proof}
We shall construct the solution $F_t$ from an unconditionally 
stable approximation scheme restricted to a discrete set of times $t_n=n\Delta t$. 
We approximate $F_{t_n}$ by measures $F_n$, discretizing the gain term
from fragmentation explicitly, and the rest of the terms implicitly.
%the coagulation terms implicitly in time and the fragmentation terms explicitly. 
We require that for every test function
$\varphi$ continuous on $[0,\infty]$, 
\begin{eqnarray}
&&\hspace{-1.5cm}
\frac 1{\Delta t}\int_{{\mathbb R}_+} \varphi(x) \, (F_{n+1}(dx) - F_{n}(dx))
 \nonumber \\
&&\hspace{-.5cm}
= \int_{{\mathbb R}_+^2} \big( \varphi (x+y) - \varphi(x) - \varphi(y) \big)  \, F_{n+1}(dx) \, F_{n+1}(dy)  \nonumber\\
&&\hspace{-.4cm}
-  \int_{{\mathbb R}_+} \varphi(x)\,F_{n+1}(dx)
+  \int_{{\mathbb R}_+} 
%\Big(  
\frac2x \int_0^x \varphi(y) \, dy 
%\Big)
\,F_n(dx) \,.
%+ \int_{{\mathbb R}_+}  
%\left( \frac2x \int_0^x  \varphi(y)  \, dy -  \int_{{\mathbb R}_+} \varphi(x) \, dF_{n}(x)
\label{eq:IEschemeF}
\end{eqnarray}

%\bigskip
%Introduce moments:
%\[
%m_k(F) = \int_{(0,\infty)} x^k\,F(dx)\,.
%\]
Taking $\varphi(x)=1-e^{-sx}$ in particular, the above scheme reduces to
the following implicit-explicit scheme for Eq.~\qref{eqC:Bernstein}
for the Bernstein transform $U_n=\breve F_n$: Given $U_n$, 
first compute
\begin{equation}\label{IE-U1}
\hat U_n(s) = U_{n}(s)   +  2\Delta t \int_0^1 U_{n}(s\tau)\, d\tau \,,
\end{equation}
then determine $U_{n+1}(s)$ by solving 
\begin{equation}\label{IE-U2}
(1+\Delta t) U_{n+1}(s) + \Delta t\, U_{n+1}(s)^2 = \hat U_n(s)\,.
\end{equation}

In order to construct a solution to the difference scheme 
\qref{eq:IEschemeF},  we first show by induction that for the solution 
of this scheme in \qref{IE-U1}-\qref{IE-U2}, 
$U_n$ is a Bernstein function for all $n\ge0$. 
Naturally, $U_0=\breve F_{\rm in}$ is Bernstein, 
so suppose $U_n$ is Bernstein for some $n\ge0$. 
Because the set of Bernstein functions is a convex cone that is 
closed under the topology of pointwise convergence,
it is clear that $\hat U_n$ is Bernstein.
Now, the function 
\[
u\mapsto v = G_{\Delta t}(u):= (1+\Delta t)u+\Delta t\,u^2
\]
is bijective on $(0,\infty)$ and its inverse is given by
\begin{equation}\label{eq:Ginv}
u =  G_{\Delta t}\inv(v) = -\alpha + \sqrt{\alpha^2+v/\Delta t},
\qquad \alpha = \frac{1+\Delta t}{2\Delta t}\,.
\end{equation}
%derivative $G_{\Delta t}'$ is Bernstein.
%By Proposition~\ref{prop:compose} the inverse function $G_{\Delta t}\inv$
This function is Bernstein, and the composition $U_{n+1}=G_{\Delta t}\inv\circ \hat U_n$
is Bernstein, by Proposition~\ref{prop:compose}.

Because $U_n(s)$ is increasing in $s$, it is clear that 
$\hat U_n(s)\le (1+2\Delta t)U_n(s)$,
and therefore that for all $n>0$ and $s\in(0,\infty]$,
\begin{equation}\label{eq:Un-bound}
U_{n}(s) \le \frac{1+2\Delta t}{1+\Delta t} U_{n-1}(s) \le e^{\Delta t} U_{n-1}(s) 
\le e^{n\Delta t} U_0(s)\,.
\end{equation}
Therefore, for each $n\ge0$, we have 
\[
U_{n}(0)=0\,, \qquad
\lim_{s\to\infty} \frac{U_{n}(s)}s =0 \,,
\]\[
U_n(\infty)\le  e^{n\Delta t}  U_0(\infty)=  e^{n\Delta t} m_0(\Finit).
\]
By \qref{eq:a0ainfty}, it follows $U_n=\breve F_n$ for some finite measure
$F_n\in \calM_+(0,\infty)$. 

We note that by the concavity of $U_n$ and the bound \qref{eq:Un-bound},
we have the following uniform bound on derivatives of $U_n$:
 for any positive $\sigma, \tau>0$, as long as $(n+1)\Delta t\le \tau$ and $s\in[\sigma,\infty)$,
\begin{equation}
U_n'(s) \le  \frac{U_n(s)}s \le e^\tau \frac{U_0(s)}s \le e^\tau\frac{U_0(\sigma)}\sigma\,,
\end{equation}
and
\begin{equation}
\frac1{\Delta t} |U_{n+1}(s)-U_n(s)| \le U_{n+1}(s)+U_{n+1}(s)^2+2U_n(s)
\le C(\tau)\,.
\end{equation}
Due to these bounds, the piecewise-linear interpolant in time is Lipschitz continuous,
uniformly for $(s,t)$ in any compact subset of $(0,\infty)\times[0,\infty)$. 
By passing to a subsequence, using the Arzela-Ascoli theorm and a diagonal extraction argument, we obtain a limit $U(s,t)$ continuous 
on $(0,\infty)\times[0,\infty)$ which 
is Bernstein for each fixed $t\ge0$ and satisfies 
the time-integrated form of Eq.~\qref{eqC:Bernstein}:
\begin{equation}\label{eq:uevolint}
\uu(s,t)=\uu_0(s) + \int_0^t \left(
 -\uu^2 -\uu + \frac2s \int_0^s \uu(r,\tau)\,dr\right)\,d\tau\,.
\end{equation}

By consequence, \qref{eqC:Bernstein} holds. 
From \qref{eq:Un-bound}, the function $U(s,t)$ inherits the bounds 
\begin{equation}
U(s,t) \le e^t U_0(s), 
\end{equation}
and we conclude 
\[
U(0,t)=0, \qquad \lim_{s\to\infty} \frac{U(s,t)}s = 0,
\]\[
 U(\infty,t)\le e^t m_0(\Finit).
\]
As previously, by \qref{eq:a0ainfty} we infer that for each $t\ge0$,
$U(s,t)=\breve F_t$ for some finite measure $F_t\in \calM_+(0,\infty)$,
and $U(\infty,t)=m_0(F_t)$.

We may now deduce \qref{eq:m0evolve} by taking $s\to\infty$ in
\qref{eq:uevolint}. (Indeed, because $U(s,t)$ increases as $s\to\infty$
toward $m_0(F_t)$, one finds $s\inv\int_0^s U(r,\tau)\,dr\to m_0(F_\tau)$
for each $\tau>0$.)
Because $t\mapsto \breve F_t(s)$
is continuous for each $s\in[0,\infty]$, we may now invoke 
Proposition~\ref{prop:narrow} to conclude that
 $t\mapsto F_t$ is narrowly continuous. 
 
 At this point, we know that \qref{eq:modelC} holds for each test function
 of the form $\varphi(x)=1-e^{-sx}$, $s\in(0,\infty]$. Linear combinations
 of these functions are dense in the space of continuous functions on 
 $[0,\infty]$. (This follows by using the 
 homeomophism $[0,\infty]\to[0,1]$ given by $x\mapsto e^{-x}$ 
 together with the Weierstrass approximation theorem.)
 Then because $m_0(F_t)$ is uniformly bounded,
 it is clear that one obtains \qref{eq:modelC} for arbitrary
 continuous $\varphi$ on $[0,\infty]$ by approximation.
% 
% Because of the estimates
% \begin{align}
% |\varphi(x+y)-\varphi(x)-\varphi(y)|\le 3\|\varphi\|
% \end{align}
% 
% \footnote{We didn't prove \qref{eq:modelC} for arbitrary $\varphi$ yet.}

It remains to prove the further statements in (i) and (ii) of the Theorem. 
Suppose  $m_1(\Finit)<\infty$. Then we claim that the first moments 
of the sequence $(F_n)$ remain constant:
\begin{equation}\label{eq:m1n}
m_1(F_n) = U_n'(0) = U_0'(0)=  m_1(\Finit)
\end{equation}
for all $n\ge0$. The proof by induction is as follows. 
Differentiating \qref{IE-U1}-\qref{IE-U2} we have
\begin{equation}\label{eq:hatUnp}
\hat U_n'(s) = U_n'(s) + 2\Delta t\int_0^1 U_n'(s\tau)\tau\,d\tau \to 
(1+\Delta t) U_n'(0)
\end{equation}
as $s\to0$, and 
\begin{equation}\label{eq:hatUnpp}
(1+\Delta t+2\Delta t\,U_{n+1}(s)) U_{n+1}'(s) = \hat U_n'(s) \to (1+\Delta t)U_{n+1}'(0).
\end{equation}
The claim follows. 

Next we claim that the first moment of the limit $F_t$ is also constant in time. 
To show this, first note that $U_n(s)/s \le U_n'(0)=m_1(\Finit).$ Taking $\Delta t\to0$ we 
infer $U(s,t)/s\le m_1(\Finit)$, and taking $s\to0$ we get 
\[
m_1(F_t)=\D_sU(0,t)\le m_1(\Finit)\,.
\]
 Next note that because $U_n'(s)$ is decreasing in $s$, 
\eqref{eq:hatUnp}-\eqref{eq:hatUnpp} imply
\begin{equation}\label{eq:Unpp}
(1+\Delta t+2\Delta t\,U_{n+1}(s)) U_{n+1}'(s) \ge (1+\Delta t)U_n'(s).
\end{equation}
Then by the bound \qref{eq:Un-bound}, for $(n+1)\Delta t\le \tau$ we have
\[
\exp(2\Delta t e^\tau U_0(s)) \, U_{n+1}'(s) \ge 
\left(1+2\Delta t\,U_{n+1}(s)\right) U_{n+1}'(s) \ge U_n'(s) \,,
\]
where the last inequality follows by dividing \qref{eq:Unpp} by $1+\Delta t$. Therefore
\begin{equation}
\exp(2\tau e^\tau U_0(s)) \, U_{n}'(s) \ge U_0'(s).
\end{equation}
This inequality persists in the limit, whence we infer by taking $s\to0$ that 
$\D_sU(0,t)\ge \D_sU_0(0)$, meaning 
$m_1(F_t) \ge m_1(\Finit)$.
This proves \qref{eq:m1const} holds.

Because $t\mapsto (x\wedge1)F_t$ is weakly continuous on $[0,\infty]$
by the continuity theorem \ref{th:Btopology},
and $m_1(F_t)$ is bounded, 
the finite measure $x\,F_t(dx)$ is vaguely continuous, 
hence it is narrowly continuous
because $m_1(F_t)$ remains constant in time.

Finally we prove (ii). The hypothesis implies that $U_0$ is a
complete Bernstein function by Theorem~\ref{thm:CBF}. 
Complete Bernstein functions form a convex cone that is closed with respect to pointwise
limits and composition, according to \cite[Cor. 7.6]{Schilling_etal_Bernstein}.

We prove by induction that $U_n$ is a complete Bernstein function for every $n\ge0$.
For if $U_n$ has this property, then clearly $\hat U_n$ does. From the formula
\qref{eq:Ginv}, $G_{\Delta t}\inv$ is a complete Bernstein function, as it is positive 
on $(0,\infty)$, analytic on $\C\setminus(-\infty,0]$ and leaves the upper half
plane invariant. Therefore $U_{n+1}$ is completely Bernstein. 

Passing to the limit
as $\Delta t\to0$, we deduce that $U(\cdot,t)$ is completely Bernstein for each
$t\ge0$. By the representation theorem \ref{thm:CBF}, we deduce that the 
measure $F_t$ has a completely monotone density as stated in (ii).

This completes the proof of Theorem~\ref{thm:EU-modelC}, except for uniqueness.
Uniqueness is proved by a Gronwall argument very similar to that used
below in section~\ref{sec:disc2cont} to study the discrete-to-continuum limit.
We refer to Eqs.~\qref{eqD:omegat}--\qref{eqD:gronwall} and omit further details.
\end{proof}

%%%%%%%%%%%%%%%%%%%%%%%%%%%%%%%%%%%%%%%%%(%%%%%%%%%%%
%%%%%%%%%%%%%%%%%%%%%%%%%%%%%%%%%%%%%%%%%%%%%%%%%%%%
\setcounter{equation}{0}
\section{Convergence to equilibrium with finite first moment} \label{sec:convergence}
%\subsection{Weak convergence to equilibrium for Model C}
\label{subsec_continuous_CVEQ}

In this section we prove that any solution with finite first moment converges 
to equilibrium in a weak sense. Our main result is the following. 
%determined by pointwise convergence of Laplace exponents. 
%
%
%
\begin{theorem}\label{thm:CVEQ1}
Suppose $F_t(dx)$ is any solution of Model C with initial data 
$F_0$ having finite first moment $m_1=m_1(F_0)$. 
Then we have
\begin{equation}\label{lim:F}
F_t(dx) \nto \frac1{m_1} f_\star(x/m_1)\,dx 
\end{equation}
and  %\quad\mbox{and}\quad 
\begin{equation}\label{lim:xF}
x F_t(dx) \nto \frac 1{m_1} x f_\star(x/m_1)\,dx
\end{equation}
narrowly, as $t\to\infty$.
%for all $s>0$, 
%we have $\uu(s,t)\to \breve f_\star(s)$ as $t\to\infty$,
%whence [Want to conclude $F_t(dx)$ and $xF_t(dx)$ converge weakly
%on $(0,\infty)$ to $f_\star(x) dx$ and $x f_\star(x) dx$ resp.]
\end{theorem}

%\noindent{\bf Proof.}  
Due to the scaling invariance of Model C,
we may assume $m_1=1$ without loss of generality.
To prove the theorem, the main step is to study Eq.~\qref{eqC:Bernstein} 
satisfied by the Bernstein transform $\uu(s,t)=\breve F_t(s)$, and prove 
\begin{equation}\label{eq:Ulim}
\breve F_t(s) \to \breve f_\star(s) \quad
\mbox{as $t\to\infty$, for all $s\in[0,\infty)$.}
\end{equation} 
%Then because $m_0(F_t)\to1$ and $m_1(F_t)\equiv1$, we will reach
%the conclusions by using the continuity theorem for Laplace transform
%and standard results related to  vague and weak convergence of measures.
We make use of the following
proposition which will also facilitate the treatment of the discrete-size 
Model D later.
\begin{proposition}\label{prop:Ulim} Let $\bar s>0$, and 
suppose $U(s,t)$ is a $C^1$ solution of \qref{eqC:Bernstein} for 
$0<s<\bar s$, $0\le t<\infty$, such that for every $t\ge0$,
$s\mapsto U(s,t)$ is increasing and concave, with 
$U(0^+,t)=0$ and $\D_sU(0^+,t)=1$. Then for all $s\in(0,\bar s)$,
\[
U(s,t) \to U_\star(s) = \breve f_\star(s) \quad\mbox{as $t\to\infty$}.
\]
\end{proposition}

%Writing $\uu(s,t)=\breve F_t(s)$, the starting point is equation 
%(\ref{eq:CF3_Niwa_Lap}), which we write in the form
%\begin{equation}
%\D_t \uu(s,t) = -\uu^2 -\uu + \frac2s \int_0^s \uu(r,t)\,dr.
%\label{eq:cv1}
%\end{equation}
%[Recall that if the initial total mass $m_1=\int_0^\infty xf_0(x)\,dx<\infty$,
%we may scale the initial data so that $m_1=1$.]
%

\begin{proof}
\noindent\textit{Step 1: Change variables.}
We write $U_0(s)=U(s,0)$ and 
\begin{equation}
\uu(s,t)= s(1-v(s,t)), \quad
\uu_\star(s)=s(1-v_\star(s)), \quad 
\uu_0(s)= s(1-v_0(s)).
\label{eq:cv:vdef}
\end{equation}
Note $s\mapsto \uu(s,t)/s$ decreases, so $s\mapsto v(s,t)$ increases
and $0<v(s,t)<1$.  Writing 
\begin{equation}
v(s,t)=v_\star(s)+w(s,t), \qquad 
v_0(s)=v_\star(s)+w_0(s),
\label{eq:cv:wdef}
\end{equation}
we have 
\begin{equation}
 -v_\star(s) \le w(s,t) \le 1-v_\star(s),
\qquad
w_0(s)=v_0(s)-v_\star(s)<v_0(s),
\label{eq:cv:wbds}
\end{equation}
for all $s,t>0$. Because $\D_t \uu_\star=0$, we find $w$ satisfies
\begin{equation}
\D_t w(s,t) =  s(w+2v_\star-2)w - w + 
\frac2{s^2} \int_0^s w(r,t)\,r\,dr.
\label{eq:cv:weq}
\end{equation}

\noindent\textit{Step 2: Upper barrier.}
Our proof of convergence is based on comparison principles,
which will be established with the aid of two lemmas.
\begin{lemma}
There exists a decreasing function $\bar b(t)\to0$ 
as $t\to\infty$ such that $\bar b(0)=1$ and that
for all $s\in(0,\bar s)$ and $t>0$,
\[
\D_t \bar w(s,t) \ge -\uu_\star(s) \bar w(s,t)\,,\quad\mbox{where}\quad
\bar w(s,t):= v_0(s)\wedge \bar b(t) \ .
%\quad\mbox{satisfies}\quad\ .
%\quad\mbox{for all $s,t>0$}.
\]
\label{lem:cv:wdefup}
\end{lemma}
%\noindent{\bf Proof of Lemma~\ref{lem:cv:wdefup}.}
\begin{proof}[Proof of Lemma~\ref{lem:cv:wdefup}.]
It suffices to choose $\bar b(t)$ to solve an ODE of the form
\begin{equation}
\D_t b(t) = -\bar J(b(t)),  \qquad \bar b(0)=1,
\end{equation}
provided the function $\bar J$ is chosen positive so that for all $s>0$, 
\begin{equation}
b<v_0(s) 
\quad\mbox{implies}\quad 
\bar J(b) \le \uu_\star(s) b.
\label{eq:cv:breq1}
\end{equation}
Let $v_0\inv$ denote the inverse function. Then $v_0\inv$ is 
increasing on its domain $[0,1)$, with the property 
\[
b<v_0(s)
\quad\mbox{implies}\quad 
v_0\inv(b)<s 
\quad\mbox{implies}\quad 
\uu_\star(v_0\inv(b))<\uu_\star(s).
\]
We may then let 
\begin{equation}
\bar J(b) = \uu_\star\circ\left( v_0\inv(b\wedge\mbox{$\frac12$})\right) \, b\,.
\end{equation}
This expression is well-defined for all $b>0$ and works as desired.
\end{proof}

\begin{lemma} For all $s\in(0,\bar s)$ and $t>0$, $w(s,t)\le \bar w(s,t)$.
\label{lem:cv:upcmp}
\end{lemma}

%\noindent{\bf Proof of Lemma~\ref{lem:cv:upcmp}.}
\begin{proof}[Proof of Lemma~\ref{lem:cv:upcmp}.]
Let $\eps>0$, $\bar s>0$. We claim $w(s,t)<\bar w(s,t)+\eps$ for all
$t\ge0$, $0\le s\le \bar s$. This is true at $t=0$ with 
$\bar b(0)=1$.  Note $w(0,t)=0$ for all $t$. 
Suppose that for some minimal $t>0$, 
\[
w(s,t)=\bar w(s,t)+\eps \quad\mbox{with $0<s\le \bar s$.}
\]
Then $\D_t w(s,t)\ge \D_t\bar w(s,t)$. However, 
because $s\mapsto \bar w(s,t)$ is increasing, we have
\[
\frac2{s^2} \int_0^s w(r,t)\,r\,dr <
\frac2{s^2} \int_0^s (\bar w(r,t)+\eps)\,r\,dr < \bar w(s,t)+\eps=w(s,t),
\]
Then 
\[
-w+
\frac2{s^2} \int_0^s w(r,t)\,r\,dr <0,
\]
so, because $0<w+v_\star\le 1$ and $s(v_\star(s)-1)=-\uu_\star(s)<0$, we deduce from (\ref{eq:cv:weq}) that
\[
\D_t w(s,t) <  s(v_\star(s)-1)\bar w = -\uu_\star(s) \bar w(s,t) \le \D_t \bar w(s,t)
\le \D_t w(s,t).
\]
This contradiction proves the claim. Because $\eps>0$ and $\bar s>0$ are 
arbitrary, the result of the Lemma follows.
\end{proof}

\noindent\textit{Step 3: Lower barrier.} In a similar way (we omit details) 
we can find $\ubar b(t)$ increasing with $\ubar b(t)\to0$ as $t\to\infty$
such that for all $s\in(0,\bar s)$ and $t>0$,
\begin{equation}
\D_t \ubar w(s,t)\le -\uu_\star(s)\ubar w(s,t),
\quad\mbox{where}\quad
\ubar w(s,t) := (-v_\star(s))\vee \ubar b(t)\,.
%\quad\mbox{satisfies }\quad\quad\mbox{for all $s,t>0$}.
\end{equation}
(Here $a\vee b$ means $\max(a,b)$.)
Then it follows $w(s,t)\ge \ubar w(s,t)$ for all $s,t>0$.

Now, because $\ubar w(s,t)\le w(s,t)\le \bar w(s,t)$ and $\ubar w(s,t)$,
$\bar w(s,t)\to0$ as $t\to \infty$ for each $s>0$, 
the convergence result for $U(s,t)$ in the Proposition follows. 
This finishes the proof of the Proposition.
\end{proof}

\begin{proof}[Proof of Theorem~\ref{thm:CVEQ1}]
%\noindent\textit{Step 4: Conclusion.}
The result of the Proposition applies to the Bernstein transform
$U(s,t)=\breve F_t(s)$ for arbitrary $\bar s>0$,
and this yields \qref{eq:Ulim}.
This finishes the main step of the proof, and it remains to deduce 
\qref{lim:F} and \qref{lim:xF}.
We know that $m_0(F_t)=\uu(\infty,t)\to1=U_\star(\infty)$
as $t\to\infty$, hence by
Proposition~\ref{prop:narrow} it follows 
$F_t(dx)\to f_\star(x)\,dx$ narrowly as $t\to\infty$.

%  %we have
%\[
%m_0(F_t)-U(s,t) = \int_0^\infty e^{-sx}F_t(dx) \to
%\int_0^\infty e^{-sx}f_\star(x)\,dx 
%%\quad\mbox{for all $s\ge0$}\,.
%\]
%for all $s\ge0$. By Proposition~\ref{prop:narrow} it follows 
%Hence, by the standard continuity theorem for Laplace transforms, 
%$F_t(dx)\to f_\star(x)\,dx$  vaguely as $t\to\infty$. But because
%$m_0(f_\star)=1=\lim m_0(F_t)$, the convergence holds narrowly [must explain and cite].

By consequence, because $m_1(F_t)\equiv 1$ is bounded,
the measures $xF_t(dx)$ converge to $xf_\star(x)\,dx$ 
vaguely, and because $m_1(F_t)\equiv m_1(f_\star)$, 
this convergence also holds narrowly.
This finishes the proof of the Theorem.
\end{proof}

%
%Provided we have a locally integrable density $f(x,t)$ defined for all $t$,
%the conclusion of the Theorem is equivalent to the statement that
%for all $z>0$,
%\begin{equation}
%\int_0^z x f(x,t) \,dx \to
%\int_0^z x f_\star(x) \,dx
%\quad\mbox{and}\quad
%\int_z^\infty f(x,t) \,dx \to
%\int_z^\infty f_\star(x) \,dx 
%\end{equation}
%as $t\to\infty$, see section 3 of [Menon-Pego 2008 J. Nonl. Sci.].
%In general, the Theorem implies weak convergence of the mass measure 
%corresponding to $x f(x,t)\,dx$ on $(0,\infty)$.

%%%%%%%%%%%%%%%%%%%%%%%%%%%%%%%%%%%%%%%%%%%%%%%%%%%%
%%%%%%%%%%%%%%%%%%%%%%%%%%%%%%%%%%%%%%%%%%%%%%%%%%%%
\section{Weak convergence to zero with infinite first moment} \label{sec:disc_zero}

Next we study the case with infinite first moment. 
If the initial data has infinite first moment, 
the solution converges to zero in a weak sense, with all clusters growing 
asymptotically to infinite size, loosely speaking.

%\subsection{Model C with infinite first moment}
\begin{theorem}
Assume $m_1(\Finit)=\int_0^\infty x \Finit(dx)=\infty$. Then 
as $t\to\infty$, 
%the solution of (\ref{eq:cv1}) satisfies $\uu(s,t)\to1$ as $t\to\infty$,
%for every $s>0$.  Moreover,
%And [explain weak convergence/Laplace transform],
\[
F_t(dx) \vto 0 %= \int_{(0,x]} F_t(dz) \to 0  
%\quad\mbox{as $t\to\infty$, }
\]
vaguely on $[0,\infty)$, while $m_0(F_t)\to1$.  
%[weak sense dual to $C_c$]
\label{thm:CVEQ2}
\end{theorem}

\noindent{\bf Proof.} The main step of the proof involves showing that
\begin{equation}\label{lim:U1}
U(s,t) \to 1 \quad\mbox{as $t\to\infty$, \  \ for all $s\in(0,\infty)$.}
\end{equation}
The assumption $m_1=\infty$ implies 
$\D_s\uu(0^+,0)=\infty$. Then for any positive constant $k$, 
we compare with a solution corresponding to $F_0(dx)=ka^2 e^{-a x}\,dx$,
whose Bernstein transform is
\[
u_0(s) = \int_0^\infty (1-e^{-sx}) ka^2 e^{-ax}\,dx = \frac{kas}{a+s} \,.
\]
Note that 
\begin{equation}
u_0(0)=0,\quad \D_s u_0(0)=k,  
\end{equation}
and for $a$ sufficiently small we have
\begin{equation}
u_0(s)<\uu_0(s) \quad\mbox{for all $s>0$}.
\label{eq:cv:uu0def}
\end{equation}
This is so because $u_0(s)\le k(s\wedge a)$,
while $U_0$ is increasing with $U_0(s)\ge ks$  for all $s$ sufficiently small.

Let $u(s,t)$ denote the solution of equation \qref{eqC:Bernstein} with
initial data $u(s,0)=u_0(s)$ for $s>0$. By a comparison argument 
similar to that above, we have 
\[
u(s,t)\le \uu(s,t) \quad\mbox{ for all $s,t>0$.}
\]
Because $u(s/k,t)$ is also a solution of \qref{eqC:Bernstein}, with first moment equal to 1,
it follows from Theorem~\ref{thm:CVEQ1}
that for every $s>0$, $u(s/k,t)\to \breve f_\star(s)$ as $t\to\infty$, 
hence due to the dilational invariance of \qref{eqC:Bernstein} we have
\begin{equation}
u(s,t)\to \breve f_\star(ks) \quad\mbox{as $t\to\infty$, for every $s>0$}.
\end{equation}
It follows that for every $s>0$,
\begin{equation}
\liminf_{t\to\infty} \uu(s,t) 
\ge \breve f_\star(ks).
\end{equation}
This is true for every $k>0$, hence we infer
\begin{equation}\label{eqC:liminfU}
\liminf_{t\to\infty} \uu(s,t) \ge 1\,.
\end{equation}
On the other hand, we know $\uu(s,t)\le m_0(t)\to1$ from 
\qref{eq:m0evolve}.
%[clean up, give reason from existence theory..].
Therefore we conclude $\uu(s,t)\to1$ as $t\to\infty$.

Because $m_0(F_t)=\uu(\infty,t)\to1$ as $t\to\infty$, we have
\[
m_0(F_t)-U(s,t) = \int_0^\infty e^{-sx}F_t(dx) \to0 \quad\mbox{for all $s\in(0,\infty)$}\,.
\]
Hence,  $F_t(dx)\vto 0$  by the standard continuity theorem for Laplace transforms.
\endproof

%%%%%%%%%%%%%%%%%%%%%%%%%%%%%%%%%%%%%%%%%%%%%%%%%%%%%%%%%%%%%%%%%%%%%%%%%%%%%%%%%%%%%%%%%%%%%%%%
%%%%%%%%%%%%%%%%%%%%%%%%%%%%%%%%%%%%%%%%%%%%%%%%%%%%%%%%%%%%%%%%%%%%%%%%%%%%%%%%%%%%%%%%%%%%%%%%
%\appendix

\section{No detailed balance for Model C} \label{sec:cts_balance}
Here we verify that the size-continuous Model C admits no equilibrium finite measure solution $F_*(dx)$ 
supported on $\R_+$ that satisfies a natural weak form of the detailed balance condition.
First, note that the detailed balance condition \qref{eq:detailed}
for the discrete coagulation-fragmentation equations can be written in a weak form
by requiring that for any bounded sequence $(\psi_{i,j})_{i,j\in\N}$
\begin{equation}
\sum_{i,j=1}^\infty \psi_{i,j}a_{i,j}f_if_j = 
\sum_{i=2}^\infty \left( \sum_{j=1}^i \psi_{i-j,j}b_{i-j,j}\right)f_i \ .
\end{equation}
For a general coagulation-fragmentation equation in the form \qref{eq:CF2}, 
we will say that $F_*$ \textit{satisfies detailed balance in weak form} 
if for any smooth bounded test function $\psi(x,y)$, 
\begin{align}
&\int_{\R_+^2} \psi(x,y)a(x,y)\, F_*(dx)\,F_*(dy) = 
\nonumber\\
&\hspace{2.5cm} \int_{\R_+} \left(\int_0^x \psi(x-y,y) b(x-y,y)\,dy\right) F_*(dx) \ .
\label{eq:detailed_balance_weak}
\end{align}
Note that if  some $F_*$ exists satisfying this condition, then it is an equilibrium solution of \qref{eq:CF2},
as one can check by taking 
\[
\psi(x,y)=\varphi(x+y)-\varphi(x)-\varphi(y)\,.
\]
\begin{theorem}
For Model C, no finite measure $F_*$ on $(0,\infty)$ exists that satisfies 
condition \qref{eq:detailed_balance_weak} for detailed balance in weak form. 
\end{theorem} 
\begin{proof} Recall that $a(x,y)=2$, $b(x,y)=2/(x+y)$ for Model C.
By taking $\psi(x,y)=1$ in \eqref{eq:detailed_balance_weak}
we find $m_0(F_*)^2 = m_0(F_*)$, so a nontrivial
solution requires $m_0(F_*)=1$.  Next, with $\psi(x,y)=1-e^{-sy}$ we find
\[
\breve F_*(s) = \int_{\R_+} \frac{e^{-sx}-1+sx}{sx}\,F_*(dx) =: \frac{G(s)}s\ ,
\] 
for all $s>0$. We compute that $G'(s)=\breve F_*(s) = G(s)/s$, whence $G(s)=Cs$ for some constant $C>0$.
But this gives $\breve F_*(s)\equiv C$, which is not possible for a
finite measure $F_*$ supported in $(0,\infty)$.
\end{proof}

\begin{remark} If detailed balance holds in the general weak form \qref{eq:detailed_balance_weak},
there is a formal $H$ theorem in the following sense:
Assume the measure $F_t(dx)$ is absolutely continuous 
with respect to $F_*(dx)$ for all $t$,
with Radon-Nikodym derivative $f_t(x)$ so that $F_t(dx)=f_t(x)F_*(dx)$.
Then the weak-form coagulation-fragmentation equation
\qref{eq:CF2} takes the form
\begin{equation}
\begin{split}
&\frac{d}{dt} \int_{\rplus} \varphi(x) \, f_t(x)\,F_*(dx) = 
\\&\ \ 
\frac{1}{2} \int_{{\mathbb R}_+^2} 
\big( \varphi (x+y) - \varphi(x) - \varphi(y) \big)
\bigl(f_t(x)f_t(y)-f_t(x+y)\bigr) K_*(dx,dy) \,, % a(x,y) F_*(dx)F_*(dy)\,.
\end{split}
\end{equation}
where $K_*(dx,dy)=a(x,y)F_*(dx)F_*(dy)$. 
If we now put
\begin{align}
{\mathcal H} &= \int_{\R_+} f_t(x)(\ln f_t(x)-1) \, F_*(dx)\,, \\
{\mathcal I} &= \frac12 \int_{\R_+^2} 
\left( \frac{ f_t(x+y)}{f_t(x)\,f_t(y)}-1\right)
\ln \frac{ f_t(x+y)}{f_t(x)\,f_t(y)} \, K_t(dx,dy)\,,
%a(x,y)\,F_t(dx)\,F_t(dy) \ .
\end{align}
where $K_t(dx,dy)=f_t(x)f_t(y)K_*(dx,dy)$,
then $\mathcal I\ge0$ and formally
\begin{equation}
\frac d{dt} {\mathcal H} + {\mathcal I} = 0\,.
\end{equation}
\end{remark}

%%%%%%%%%%%%%%%%%%%%%%%%%%%%%%%%%%%%%%%%%%%%%%%%%%%%%%%%%%%%%%%
%%%%%%%%%%%%%%%%%%%%%%%%%%%%%%%%%%%%%%%%%%%%%%%%%%%%%%%%%%%%%%%
%%%%%%%%%%%%%%%%%%%%%%%%%%%%%%%%%%%%%%%%%%%%%%%%%%%%%%%%%%%%%%%
\vfil\pagebreak

\part{Analysis of Model D}
\myrunningheads

\section{Equations for the discrete-size model}

As discussed in section 2.3, we take $\alpha=\beta=2$ in the expression 
\qref{eq:rates_niwa_disc} 
for the coagulation and fragmentation rates to get the equations for Model D
that will be studied below. 
In weak form we require that for any bounded test sequence $(\varphi_i)$, % [with limit?],
\begin{eqnarray}
&&\hspace{-1cm}
\frac{d}{dt} \sum_{i=1}^\infty  \varphi_i \, f_i(t) =
\sum_{i,j=1}^\infty \big( \varphi_{i+j} - \varphi_i - \varphi_j \big) \,   f_i(t) \, f_j(t) \nonumber \\
&&\hspace{2.cm}
+ \sum_{i=1}^\infty \Big( -  \varphi_i + \frac{2}{i+1} \sum_{j=1}^{i} \varphi_j  \Big) f_i(t) \,  . 
\label{eq:CF2_disc_NiwaC}
\end{eqnarray}
In strong form, the system is written as follows: 
\begin{eqnarray}
&&\hspace{-1.5cm}
    \frac{\partial f_i}{\partial t}(t) = Q_a(f)_i(t) + Q_b(f)_i(t) ,
\label{eq:CF3_disc_NiwaC}\\
&&\hspace{-1.5cm}
Q_a(f)_i(t) = \sum_{j=1}^{i-1} f_j(t)  \, f_{i-j}(t) - 2 \sum_{j=1}^\infty f_i(t) \, f_j(t) , 
\label{eq:CF4_disc_NiwaC} \\
&&\hspace{-1.5cm}
Q_b(f)_i(x,t) =  -f_i(t) + 2 \sum_{j=i}^\infty \frac{1}{j+1} \, f_{j}(t) \nonumber\\
&&\hspace{-0.2cm}
\qquad = -\left(\frac{i-1}{i+1}\right)  f_i(t) + 2 \sum_{j=i+1}^\infty \frac{1}{j+1} \, f_{j}(t). 
\label{eq:CF5_disc_NiwaC}
\end{eqnarray}

\textit{Bernstein transform.} %\label{sec:disc_Bernstein_transform}
It is useful to describe the long-time dynamics of the discrete model in terms of the Bernstein
transform of the discrete measure $\sum_{j=1}^\infty f_j(t)\,\delta_j(dx)$. 
We define
\begin{equation}\label{eqD:BTdef}
\breve f(s,t) = \sum_{j=1}^\infty (1-e^{-js}) f_j(t) \ .
\end{equation}
Taking the test function $\varphi_j = 1-e^{-js}$ in \qref{eq:CF2_disc_NiwaC}
and using the fact that
\[
\varphi_{i+j}-\varphi_i-\varphi_j = -\varphi_i\varphi_j \ ,
\]
we find that $\breve f(s,t)$ satisfies the equation
\begin{equation}\label{eqD:BT1}
\boxed{ 
\D_t \breve f(s,t) = -\breve f^2 - \breve f + 2A_1(\breve f) \ ,
}\end{equation}
for any $s, t>0$, where
\begin{eqnarray}
 A_1(\breve f)(s,t) &=& \sum_{i=1}^\infty \frac{f_i(t)}{i+1} \sum_{j=0}^i (1-e^{-sj})
%= m_0(f) - \sum_{i=1}^\infty \frac{f_i(t)}{i+1} \frac{1-e^{-s(i+1)}}{1-e^{-s}} %\\
\nonumber\\ &=&  
 \sum_{i=1}^\infty f_i(t) \left(1- \frac{1}{1-e^{-s}} \int_0^s  e^{-r(i+1)}\,dr \right)
\nonumber \\ &=& %\quad %\hspace{1.4cm}
  \sum_{i=1}^\infty \frac{f_i(t)}{1-e^{-s}} 
 \int_0^s (1-e^{-ri})e^{-r}\,dr  
\nonumber \\ &=& %\quad %\hspace{1.4cm}
  \frac1{1-e^{-s}}\int_0^s \breve f(r,t) e^{-r}\,dr.
\end{eqnarray}
Eq.~\qref{eqD:BT1} is a nonlocal analog of the logistic equation 
that one would obtain if the averaging term $2A_1(\breve f)$ were replaced by $2\breve f$.

It is very remarkable that Eq.~\qref{eqD:BT1} for the Bernstein
transform of Model D transforms \textit{exactly} into Eq.~\qref{eqC:Bernstein}
for the Bernstein transform of Model C, by a simple change of variables. With 
\begin{equation}\label{eqD:change}
\sigma = 1-e^{-s}\,, \qquad u(\sigma,t)= \breve f(s,t)\,,
\end{equation}
one finds that \qref{eqD:BT1} for $s\in(0,\infty)$, $t>0$, is equivalent to 
\begin{equation}\label{eqD:BTueq}
\D_t u(\sigma,t) = -u^2-u + \frac2\sigma\int_0^\sigma u(r,t)\,dr\,,
\end{equation}
for $\sigma\in(0,1)$, $t>0$.  This equation will be used below in 
sections~\ref{secD:equilibrium} and \ref{secD:longtime}
below to study the equilibria and long-time behavior of solutions to Model D.

%%%%%%%%%%%%%%%%%%%%%%%%%%%%%%%%%%%%%%%
%%%%%%%%%%%%%%%%%%%%%%%%%%%%%%%%%%%%%%%
\section{Equilibrium profiles for Model D} \label{secD:equilibrium}

As will become clear from the well-posedness theory to come,
any equilibrium solution $f= (f_i)$ of Model D has a finite zeroth moment 
\[
m_0( f) =  \sum_{i=1}^\infty f_i \ .
\]
\begin{theorem}\label{thmD:equilibria}
For every $\mu\in[0,\infty)$ there is a unique equilibrium solution 
$f_\mu$ of Model D such that 
\[
m_1(f_\mu)= \sum_{i=1}^\infty i f_\mui = \mu \, . 
\]
The solution has the form 
\begin{equation}\label{eqD:lambdamu}
f_\mui = \gamma_\mui\, \lambda_\mu^{-i}  \ ,  \qquad 
\lambda_\mu = 1 + \frac4{27\mu}\,,
\end{equation}
where $\gamma_\mu$ is a completely monotone sequence with the asymptotic 
behavior
\begin{equation}\label{eqD:gammalim}
\gamma_\mui \sim \frac98 
\left( \frac{\mu\lambda_\mu}{\pi} \right)^{1/2} 
i^{-3/2} 
\quad\mbox{as $i\to\infty$}.
\end{equation}
Every equilibrium solution has the form $f_\mu$ for some $\mu$.
\end{theorem}

\begin{proof}
To start the proof of the theorem, let $f=(f_i)$ be a nonzero equilibrium
solution of Model D with Bernstein transform $\breve f(s)$. 
Change variables to  $\sigma=1-e^{-s}\in(0,1)$ as in \qref{eqD:change} and
introduce $u(\sigma)=\breve f(s)$, $\sigma\in(0,1)$. Then $u$ is a stationary solution
of \qref{eqD:BT1} and satisfies \qref{eq:Uint1}, 
and it follows that 
\begin{equation}\label{eqD:usigma}
u(\sigma)=U_\star(\mu \sigma) \,, 
\quad\mbox{or}\quad
\breve f(s)= U_\star(\mu-\mu e^{-s})\,,
\end{equation} 
for some 
$\mu\in(0,\infty)$. Then by differentiation it follows that 
\begin{equation}\label{eqD:m1def}
\mu  = u'(0) = \breve f'(0) = \sum_{i=1}^\infty i f_i = m_1(f).
\end{equation}
Note that by \qref{eq:Usolution}  the zeroth moment 
of $f$ satisfies 
\begin{equation}\label{eqD:m0m1}
m_0(f)=\breve f(\infty)=U_\star(\mu)\,,\qquad 
\frac{m_0(f)}{(1-m_0(f))^3} = \mu.
\end{equation}

%\subsection{Generating function of Model D equilbria}
{\sl Generating function.} 
The properties of the equilibrium sequence $f$ shall be derived from the 
behavior of the {\sl generating function}
\begin{equation} \label{d:modelDgen}
G(z) := \sum_{i=0}^\infty f_i \,z^i \, ,
\end{equation} 
where we find it convenient (see section~\ref{secDC:discretization} below) to define 
\begin{equation}\label{eqD:f0}
 f_0 = 1-m_0(f) = 1-U_\star(\mu) \ .
\end{equation}

{\sl Complete monotonicity.}
Due to \qref{eqD:usigma} and \qref{eqD:f0}, we have
\[
U_\star(\mu\sigma) =\breve f(s)  = 1-G(e^{-s}) =  1-G(1-\sigma)\,.
\]
Hence by \qref{eqC:B3U}, we obtain
\begin{equation}\label{eqD:GB3}
\boxed{G(z) = B_3(\mu(z-1)).}
\end{equation}
By Lemma~\ref{L:B3} (from \cite{LP2016}), $G$ is a Pick function
analytic and nonnegative on the interval $(-\infty,\lambda_\mu)$ where
$\lambda_\mu=1+ \frac4{27\mu}$. 
Nonnegativity, and indeed complete monotonicity, of the sequence 
$\gamma_\mu$ given by 
\begin{equation}\label{eqD:gamma}
\gamma_\mui = f_i\,\lambda_\mu^i \ ,
\end{equation}
now follows immediately from Theorem~\ref{thm:Pick}.

%This was one of
%the main results of \cite{LP2016}, and it provides a discrete analog of 
%the representation theorem \ref{thm:CBF} above, taken from \cite{Schilling_etal_Bernstein}.

%Note from (\ref{eq:FUrel}) and \qref{eq:m0modelD} we have
%\[
%1-G(e^{-s}) = \breve F(s) =  U(m_1^h z) \,,
%\]
%%\[
%%m_0^h-F(e^{-sh}) = \breve F(s) = U(m_1^h z) \,,
%%\]
%therefore, with $x=e^{-sh} = 1-hz$ we have
%\begin{equation}\label{eq:FFUrel}
%\boxed{F(x) = 1 - U((1-x)m_1^h/h) = B_3((x-1)m_1^h/h)\,.}
%\end{equation}
%%\begin{equation}\label{eq:FFUrel.o}
%%\boxed{F(x) = m_0^h - U((1-x)m_1^h/h) \,.}
%%\end{equation}
%
%{\sl Complete monotonicity.}  As is proved in Lemma 3 of \cite{LP2016}, $B_3$ is a 
%Pick function analytic and nonnegative on $(-\infty, \frac4{27})$. 
%Consequently, $F$ is a Pick function analytic and nonnegative on $(-\infty,\hat x)$
%where
%\[
%\hat x = 1+ \frac{4h}{27m_1^h}\,.
%\]
%As a direct consequence of Corollary 1 of \cite{LP2016}, we infer that the sequence
%$(f_n^h\hat x^n)_{n\ge0}$ is a completely monotone sequence. [A definition of this
%needs to be recalled above.]  Because the sequence $(\hat x^{-n})$ is completely monotone,
%it follows that $(f_n^h)_{n\ge0}$ is completely monotone, since it is the pointwise
%product of two completely monotone sequences.

{\sl Asymptotics.}
The decay rate of the sequence $f$ shall be deduced from the derivative
of $G(z)$ using Tauberian arguments, as developed in the book of 
Flajolet and Sedgewick~\cite{FlajoletS}.

Recall that $U_\star(s)$ has a branch point at $s=-\frac4{27}$, with 
$U_\star'(s-\frac4{27}) \sim \frac98 s^{-1/2}$ due to \qref{eqC:Vz0}.
The generating function $G(z)$ has a corresponding branch point at $z=\lambda_\mu$.
% where
%\[
%\frac{m_1^h(1-\hat x)}{h} =-\frac4{27}\,,\qquad \hat x = 1+ \frac{4h}{27m_1^h}\,.
%\]

We rescale, replacing $z$ by $\lambda_\mu z$ to write
\[
G(\lambda_\mu z) =
\sum_{i=0}^\infty \gamma_\mui z^i 
= 1 - U_\star(\mu(1-\lambda_\mu z)) %\left(\frac{m_1^h}h (1-\hat xx)\right) \,.
\]
Then differentiate, writing $\hat\lambda = \mu\lambda_\mu$, %\hat x m_1^h/h$,
 to find that
\begin{align}
\sum_{i=1}^\infty i\gamma_\mui z^{i-1} 
&= \hat\lambda U_\star'(\mu(1-\lambda_\mu z)) 
 = \hat\lambda U_\star'\left(\hat\lambda(1-z)-\frac4{27}\right)
 \nonumber \\
& \sim \frac{9\hat\lambda^{1/2}}8 (1-z)^{-1/2} 
%\sum_{n=1}^\infty n f_n^h\hat x^n x^{n-1} 
%&= \hat\lambda\, U'\left((1-\hat xx)\frac{m_1^h}h\right) 
% = \hat\lambda\, U'\left(\hat\lambda(1-x) -\frac4{27} \right) 
%  \sim \frac{9\hat\lambda^{1/2}}8 (1-x)^{-1/2} 
\label{eq:diffF1}
\end{align}
By Corollary VI.1 from \cite{FlajoletS} we deduce that as $i\to\infty$,
\[
i \gamma_\mui  \sim 
\frac{9\hat\lambda^{1/2}}8 \frac{(i-1)^{-1/2}}{\Gamma(1/2)} 
\sim
\frac{9\hat\lambda^{1/2}}8 \frac{i^{-1/2}}{\Gamma(1/2)},
\]
This yields \qref{eqD:gammalim}, since $\Gamma(1/2)=\sqrt\pi$.
\end{proof}
%\begin{equation}\label{eq:fnasymptotics}
%\boxed{f_n^h \sim \frac98 \left(\frac{m_1^h \hat x}{h\pi}\right)^{1/2} n^{-3/2} \hat x^{-n} \ ,
%\qquad \hat x = 1+ \frac{4h}{27m_1^h}\,.
%}\end{equation}

%%%%%%%%%%%%%%%%%%%%%%%
\subsection{Recursive computation of equilibria for Model D}
At equilibrium, the profile is required to satisfy
\begin{equation}
0= \sum_{i,j = 1}^\infty (\varphi_{i+j} - \varphi_i - \varphi_j) \, f_i \, f_j 
+ \sum_{i = 1}^\infty \big( - \varphi_i+ \frac{2}{i+1} \sum_{j = 1}^{i} \varphi_j  \big)f_i \ .
\end{equation}
Define
\begin{equation}\label{e:bidef}
 \nu_0= m_0(f) = \sum_{j=1}^\infty f_j \ , \qquad 
 \beta_i= \sum_{j=i}^\infty \frac{1}{j+1} \, f_{j} \ .
\end{equation}
Taking $\varphi_j\equiv1$ yields
\begin{equation}\label{e:f1eq}
0 = -\nu_0^2 - \nu_0 + 2 \sum_{i=1}^\infty \frac{i}{i+1} f_i
= -\nu_0^2 + \nu_0 -2\beta_1 \ .
\end{equation}
Next, taking $\varphi_k=1$ if $k=i$ and 0 otherwise yields
\begin{equation}\label{e:eqstrong}
0= \sum_{j=1}^{i-1} f_j  \, f_{i-j} - (2 \nu_0+1)f_i + 2 \beta_i, \qquad i\ge1.
\end{equation}

Based on these equations, we can compute the $f_j$ recursively,
in a manner analogous to the computation
of equilibria in the models studied by Ma et al.\ in \cite{Ma_etal_JTB11}.
 Starting from any given value of 
the parameter $\nu_0\in(0,1)$, 
set $\beta_1$ according to \qref{e:f1eq},   defining
\begin{equation}\label{e:f1rec}
\beta_1=\frac12(\nu_0-\nu_0^2) \ .
\end{equation}
Then for $i=1,2,3,\dots$ compute
\begin{eqnarray}\label{e:fjrec}
 f_i &=& 
\left(1+2\nu_0 \right)\inv \left(2 \beta_{i} + \sum_{j=1}^{i-1} f_j   f_{i-j}\right), \\
 \beta_{i+1} &=& \beta_{i} - \frac{ f_{i} }{ i+1}.
\label{e:brec}
\end{eqnarray}
These formulae are used to compute the discrete profile that is compared
to the continuous profile $f_\star$ in Figure~2 in section 15 below.
%Hmmm??.\footnote{Show some numerical results for different $\nu_0$?}ß†

%\vfil\pagebreak
%%%%%%%%%%%%%%%%%%%%%%%%%%%%%%%%%%%%
%%%%%%%%%%%%%%%%%%%%%%%%%%%%%%%%%%%%
\section{Well-posedness for Model D} \label{sec:disc_wellposed}
Here we consider the discrete dynamics described by 
Eqs.~(\ref{eq:CF2_disc_NiwaC})--(\ref{eq:CF5_disc_NiwaC}), 
and establish well-posedness of the initial-value problem
by a simple strategy of proving local Lipschitz estimates on an appropriate Banach space.

We first introduce some preliminary notations. For $k \in {\mathbb R}$, 
and a real sequence $f = (f_i)_{i=1}^\infty$, 
we define the $k$-th moment $m_k(f)$ of $f$ and associated norm $\|f\|_k$ by:
\[
 m_k(f) = \sum_{i=1}^\infty i^k \, f_i \ , \qquad
 \|f\|_k = \sum_{i=1}^\infty i^k\, |f_i|\ .
 \]
We introduce the Banach space $\ell_{1,k}$ as the vector
space of sequences $f$ with finite norm $\|f\|_k<\infty$ 
and denote the positive cone in this space by 
\[
 \ell_{1,k}^+ = \big\{ f \in\ell_{1,k} \, \mbox{ such that } 
 f_i \geq 0  \mbox{ for all $i$} \big\}.
 \]

\begin{theorem}\label{t:disc_ivp}
Let $k\ge0$ and let $\finit = (f_{{\rm in}, i})_{i=1}^\infty$ be given in $\ell_{1,k}^+$. Then there exists a unique global-in-time solution 
$f\in C^1([0,\infty),\ell_{1,k}^+)$ for system 
(\ref{eq:CF2_disc_NiwaC})--(\ref{eq:CF5_disc_NiwaC})
%(\ref{eq:CF3_disc}) 
with initial condition $f(0)=\finit$. 
Moreover, $f$ is $C^\infty$, and for all $t\ge0$ we have
\begin{equation}\label{eqD:m0ineq}
\D_t m_0(f(t))\le -m_0(f(t))^2 + m_0(f(t))\,,
\end{equation}
and
\begin{equation} \label{eqD:m0bound}
m_0(f(t))\le m_0(\finit) +1\,. %\, \quad \mbox{}\;.
\end{equation}
If $k\ge1$, then 
%$m_1(\finit)<\infty$, then $f\in C^1([0,\infty),\ell_{1,1})$ and 
$m_1(f(t))=m_1(\finit)$ for all $t\ge0$.
\end{theorem}

\begin{proof} We first transform the problem to simplify the proof of positivity. 
Fix $\lambda= 2m_0(\finit)+3$ and change variables using
$\hat f_i(t) = e^{\lambda t}f_i(t)$. Then Eq.~\qref{eq:CF2_disc_NiwaC} is equivalent to 
\begin{equation}\label{eqD:hatf}
\frac{\D \hat f_i}{\D t} = 
Q_\lambda(\hat f)_i := \lambda \hat f_i + Q_a(\hat f)_i e^{-\lambda t} + Q_b(\hat f)_i.
\end{equation}
Using the inequality $i^k\le 2^k(j^k+(i-j)^k)$, we find that the quadratic map
$f\mapsto Q_a(f)$ is locally Lipschitz on $\ell_{1,k}$, satisfying 
\begin{eqnarray}
2^{-k-1}\|Q_a(f)-Q_a(g)\|_k &\le& 
\|f - g\|_k(\|f\|_0+\|g\|_0) 
\nonumber\\
&&
+\ \|f-g\|_0(\|f\|_k+\|g\|_k)\,.
\end{eqnarray}
Also, the linear map $f\mapsto Q_b(f)$ is bounded on $\ell_{1,k}$, due to 
the estimate
\[
\sum_{i=1}^\infty i^k \sum_{j=i}^\infty \frac1{j+1}|f_j| 
= \sum_{j=1}^\infty \sum_{i=1}^j \frac{i^k}{j+1}|f_j| 
\le \sum_{j=1}^\infty j^k|f_j| \ .
\]
Consequently, we have local existence of a unique smooth solution $\hat f$
to \qref{eqD:hatf} with values in $\ell_{1,k}$, and a corresponding 
smooth solution $f$ to  \qref{eq:CF2_disc_NiwaC}. 
If $k\ge1$, then because $f\mapsto m_1(f)$ is bounded on $\ell_{1,k}$,
we find $m_1(f(t))$ is constant in time
by taking $\varphi_i=i$ in \qref{eq:CF2_disc_NiwaC}. 
   
It remains to prove positivity and global existence, for any $k\ge0$.
The solution $\hat f$ is the limit of Picard iterates $\hat f\supn$ starting
with $\hat f^{(0)}(t)\equiv \finit$.
Note that the coefficient of $\hat f_i$ in $Q_\lambda(\hat f)_i$ is,
from \eqref{eq:CF4_disc_NiwaC}-\eqref{eq:CF5_disc_NiwaC},
\[
\lambda-2m_0(\hat f)e^{-\lambda t}-\frac{i-1}{i+1} > 
2(\hat M-m_0(\hat f)e^{-\lambda t}), \qquad \hat M:=\frac{\lambda-1}2.
 %\lambda-1-2m_0(\hat f)e^{-\lambda t} \,.
\]

%Let $\alpha_0=m_0(\finit)+1=\frac12\lambda$. 
Let $\nu_n(t)=m_0(\hat f\supn(t))e^{-\lambda t}$. 
We show by induction the following statement: For all $n\in\N$, 
for all $t\ge0$ in the interval of existence,  
\begin{equation}\label{n-hyp}
0\le \nu_n(t)\le \hat M
\quad\mbox{and}\quad 
\hat f_i\supn(t)\ge0 
\quad\mbox{for all $i$.}
\end{equation}
For the induction step,  note that by the induction hypothesis, 
$Q_\lambda(\hat f\supn)_i\ge0$ for all $i$, therefore $\hat f^{(n+1)}_i(t)\ge0$
for all $i$ and $t\ge0$. 
We also have
\[
\sum_{i=1}^\infty Q_a(\hat f\supn(t))_i e^{-\lambda t} = -\nu_n(t)^2e^{\lambda t}, 
 \qquad \sum_{i=1}^\infty Q_b(\hat f\supn(t))_i \le \nu_n(t)e^{\lambda t}\,,
\]
hence
\begin{equation}
e^{-\lambda t} \D_t  (e^{\lambda t}\nu_{n+1})=
e^{-\lambda t} \D_t  m_0(\hat f^{(n+1)}) \le 
\lambda \nu_n - \nu_n^2 +\nu_n \le \lambda \hat M\,.% \alpha_0\,.
%\begin{eqnarray}
%\D_t \hat m^{(n+1)} &\le& \lambda\hat m\supn - (\hat m\supn)^2 e^{-\lambda t}+\hat m\supn \nonumber \\
%e^{-\lambda t} \D_t  m_0(\hat f^{(n+1)}) \le 
%\lambda \hat m_n - \hat m_n^2 +\hat m_n \le \frac12\lambda^2\,.% \alpha_0\,.
\end{equation}
%\\
%&= & (\lambda + 1-e^{-\lambda t}\hat m\supn )\hat m\supn \le %\lambda\alpha_0e^{\lambda t}. 
%\end{eqnarray}
The last inequality holds because $x\mapsto \lambda x-x^2+x$ increases
on $[0,\hat M]$ and $\hat M>1$.
Upon integration we deduce $0\le\nu_{n+1}(t)\le \hat M$.

Taking $n\to\infty$, we find that $\hat f_i(t)\ge0$ and 
$m_0(\hat f(t))\le \hat M e^{\lambda t}$, hence
\[
f(t)=e^{-\lambda t}\hat f(t) \in \ell_{1,k}^+ \quad\mbox{and}\quad
m_0(f(t))\le \hat M
\]
on the maximal interval of existence.
 Due to the local Lipschitz bounds established above,
the solution can be continued to exist in $\ell_{1,k}^+$ for all $t\in[0,\infty)$. 

To obtain \qref{eqD:m0ineq}, take $\varphi_i\equiv 1$ in 
\qref{eq:CF2_disc_NiwaC}, or take $s\to\infty$ in \qref{eqD:BT1}.
\end{proof}

%\vfil\pagebreak
%%%%%%%%%%%%%%%%%%%%%%%%%%%%%%%%%%%%%%%%%%%%%%%%%%%%
%%%%%%%%%%%%%%%%%%%%%%%%%%%%%%%%%%%%%%%%%%%%%%%%%%%%
\section{Long-time behavior for Model D}
\label{secD:longtime}

By using much of the same analysis as in the continuous-size case,
we obtain strong convergence to equilibrium for solutions with a finite first moment,
and weak convergence to zero for solutions with infinite first moment.

%\subsection{Strong convergence to equilibrium with finite first moment}

\begin{theorem}[Strong convergence with finite first moment]\label{thmD:CVEQ1}
Let $f(t)=(f_i(t))$ be any solution of Model D with initial data 
having finite first moment $\mu=m_1(\finit)$, so $\finit\in\ell_{1,1}^+$. 
Then the solution  converges strongly to the equilibrium solution $f_\mu$
having the same first moment:
\begin{equation}\label{limD:F}
\|f(t)-f_\mu\|_1 = \sum_{i=1} i|f_i(t)-f_\mui| \to 0 \quad\mbox{as $t\to\infty$.}
\end{equation}
\end{theorem}

\begin{proof}
Let $\breve f(s,t)$ from \qref{eqD:BTdef} be the Bernstein transform of the solution,
and let $\mu=m_1(f)$ be the first moment (constant in time). Similarly to \qref{eqD:change},
we change variables according to 
 \begin{equation}\label{eqD:change2}
\sigma = \mu(1-e^{-s})\,, \quad u(\sigma,t)= \breve f(s,t)
= \sum_{i=1}^\infty \left(1-\left(1-\frac\sigma\mu\right)^i\right)f_i(t)\,.
\end{equation}
This function $u(\sigma,t)$ then is a solution to \qref{eqD:BTueq} 
for $0<\sigma<\mu$ and has the properties
that for all $t>0$, $u(0,t)=0$ and $\sigma\mapsto u(\sigma,t)$
is increasing and concave.

We now invoke Proposition~\ref{prop:Ulim} with $\bar s=\mu$,
and conclude that as $t\to\infty$,
$u(\sigma,t)\to U_*(\sigma)$ for all $\sigma\in(0,\mu)$.
It follows that 
\[
\breve f(s,t)\to \breve f_\mu(s) \quad\mbox{for all $s>0$}.
\]
By the continuity theorem \ref{th:Btopology}, it follows that as $t\to\infty$,
the discrete measures
\[
F_t(dx)=\sum_i f_i(t)\delta_i(dx) \wkto F_\mu(dx)=\sum_i f_\mui\delta_i(dx)
\]
weakly on $[0,\infty]$. But this implies $f_i(t)\to f_\mui$ for every $i\in\N$.
Then \qref{limD:F} follows, see \cite[Lemma 3.3]{BallCarrPenrose1986},
for example.
\end{proof}
%[Exponential convergence rate with a boundedness condition?? on moments??]

%\subsection{Weak convergence to zero with infinite first moment}
%his should work the same as in the continuous case. 
\begin{theorem}[Weak convergence to zero with infinite first moment]
Let $f(t)$ be any solution of Model D with initial data 
$\finit\in\ell_{1,0}$ having infinite first moment. 
Then as $t\to\infty$ we have
$f_i(t)\to0$ for all $i$, and 
\[
m_0(f(t)) = \sum_{i=1}^\infty f_i(t) \to 1.
\]
\end{theorem}

\begin{remark}
The conclusion means that the total number of groups $m_0(f(t))\to1$,
while the number of groups of any fixed size $i$ tends to zero.  
Thus as time increases, individuals cluster in larger and larger groups,
leaving no groups of finite size in the large-time limit.
\end{remark}

\begin{proof}
We use the change of variables in \qref{eqD:change} and obtain
a solution $u(\sigma,t)$ to \qref{eqD:BTueq} that satisfies 
$\D_\sigma u(0^+,0)=\infty$.  By the same arguments as in
the proof of Theorem~\ref{thm:CVEQ2}, we obtain the analog of
\qref{eqC:liminfU}, namely
\begin{equation}\label{eqD:liminfu}
\liminf_{t\to\infty} u(\sigma,t) \ge 1 \quad\mbox{for all $\sigma\in(0,1)$.}
\end{equation}
We know, though, that $u(\sigma,t)\le m_0(f(t))$ and that 
\[
\limsup_{t\to\infty} m_0(f(t))\le1
\]
due to \qref{eqD:m0ineq}. It follows that $m_0(f(t))\to 1$ and
$\breve f(s,t)\to1$ for all $s>0$.
Therefore the Laplace transform $\sum_{i=1}^\infty e^{-si}f_i(t)\to0$,
and the conclusions of the Theorem follow. 
\end{proof}
%%%%%%%%%%%%%%%%%%%%%%%%%%%%%%%%%%%%%%%%%%%%%%%%%%%%
%%%%%%%%%%%%%%%%%%%%%%%%%%%%%%%%%%%%%%%%%%%%%%%%%%%%
%%%%%%%%%%%%%%%%%%%%%%%%%%%%%%%%%%%%%%%%%%%%%%%%%%%%
%%%%%%%%%%%%%%%%%%%%%%%%%%%%%%%%%%%%%%%%%%%%%%%%%%%%
\vfil\pagebreak

\part{From discrete to continuous size}
\myrunningheads

\section{Discretization of Model C}\label{secDC:discretization}

In this section, we discuss how  a particular discretization of 
Model C is naturally related to Model D. 
We start from the weak form of Model C expressed in \qref{eq:modelC}. 
Introduce a grid size $h>0$, corresponding to a scaled size increment, 
and introduce the approximation
\begin{equation}
f_i^h \approx \int_{I^h_i} F_t(dx) \,, \qquad I^h_i:=(ih,(i+1)h],
\quad i=0,1,\ldots
\end{equation}
for  the number of clusters with size in the interval $I^h_i$. 
(We find it convenient to not scale this number by the width of $I^h_i$,
for purposes of comparison.)

Then by formal discretiztion of the integrals in \qref{eq:modelC} 
by the left-endpoint rule, we require 
\begin{eqnarray}
&& \hspace{-1cm}
\sum_{i=0}^\infty \varphi(ih) \frac{d f_i^h}{dt} (t) = \sum_{i,j=0}^\infty \Big( \varphi \big( (i+j)h \big) -  \varphi (ih) - \varphi(jh) \Big) \, f_i^h(t) \, f_j^h(t) 
\nonumber \\
&& \hspace{2cm} 
+ \sum_{i=0}^\infty \Big(  -  \varphi (ih)+\frac{2}{i+1} \sum_{j=0}^{i} \varphi(jh) \Big) \, f_i^h(t). 
\label{eqDC:weak}
\end{eqnarray}

{\sl ``Ghost'' clusters.} 
Because of the left-endpoint discretization, the sums in \qref{eqDC:weak}
start with $i,j=0$. Note, however, that if we take
$\varphi(0)=0$ and $\varphi_i=\varphi(ih)$, all terms with $i=0$ or $j=0$ drop,
and Eq.~\qref{eqDC:weak} becomes identical to the weak form of Model D
in \qref{eq:CF2_disc_NiwaC}. Thus, the dynamics of the
sequence $(f^h_i(t))_{i=1}^\infty$ is governed exactly by Model D, 
and is decoupled from the behavior of $f^h_0(t)$.
The equation governing $f^h_0$ corresponds to the coefficient of $\varphi(0)$
in Eq.~\qref{eqDC:weak} and takes the form
\begin{equation}\label{eqDC:f0t}
\D_t f^h_0(t) = - (f^h_0)^2 - 2 f^h_0\sum_{i=1}^\infty f^h_i + f^h_0 + 
2\sum_{i=1}^\infty \frac1{i+1}f^h_i\ .
\end{equation}
We see the behavior of $f^h_0(t)$ is slaved to that of $(f^h_i(t))_{i=1}^\infty$.

In the discrete-size model, $f^h_0$ has the interpretation as number density
of clusters whose size is less than the bin width $h$.  
Merging and splitting interactions with such clusters have a negligible effect upon 
dynamics in the discrete approximation, so we can say these clusters become 
``ghosts.''  
It is sometimes convenient, however, to include them in the tally of total cluster 
numbers, for the following reason:
Taking $\varphi\equiv1$ in \qref{eqDC:weak} we find that the quantity
\[
\hat\nu_0(t):=
m_0\left((f_i^h(t))_{i=0}^\infty\right) = 
\sum_{i=0}^\infty f^h_i(t)
\]
is an exact solution of the logistic equation:
\begin{equation}
\D_t \hat\nu_0(t) = -\hat\nu_0^2+\hat\nu_0.
\end{equation}
(Recall that without the $i=0$ term, we have only the inequality \qref{eqD:m0ineq}.)
Naturally $\hat\nu_0=1$ in equilibrium, which helps to explain 
the convenient choice of $f_0$ in \qref{eqD:f0}: 
The function $G$ in \eqref{d:modelDgen} is the generating function of 
$(f_i)_{i=0}^\infty$, an equilibrium for Model D extended to include ghost clusters.

%We remark\footnote{remark??} that discretiztion of Model C by the right-endpoint rule results
%in a time-continuous model with coefficients xxx...  [[Not Qi Ma?]]

{\sl Bernstein transform.}
By taking $\varphi(x)=1-e^{-sx}$, we obtain the $h$-scaled Bernstein transform
\begin{equation}
U^h(s,t) = \sum_{i=1}^\infty (1-e^{-shi})f_i^h(t) = \breve F_t^h(s)\,,
\end{equation}
where $F^h_t$ is the discrete measure on the grid $\{ih: i=1,\ldots\}$
formed from the solution $f(t)=(f^h_i(t))_{i=1}^\infty$ of Model D:
\[
F^h_t(dx) = \sum_{i=1}^\infty f^h_i(t) \,\delta_{ih}(dx)\ .
\]
(Ghost clusters would have no effect on $U^h$, and
we do not include them here, in order to focus on how Model D compares to Model C.)
The function $U^h$ satisfies the following scaled variant of 
\qref{eqD:BT1}:
\begin{equation}
%\boxed{\D_t \breve F_t^h(s) = -\breve F_t^h(s)^2 - \breve F_t^h(s) + 2 A_h(\breve F_t^h)(s)\,,}
\boxed{\D_t U^h(s,t) = -(U^h)^2 - U^h + 2 A_h(U^h)\,,}
\label{eq:BTh}\end{equation}
where the scaled averaging operator
\begin{equation}
A_h(U^h)(s,t) = \frac{h}{1-e^{-sh}} \int_0^s U^h(r,t) e^{-rh}\,dr \ .
\label{def:Ah}
\end{equation}
In the limit $h\to0$, the operator $A_h$ reduces formally to the running average
operator $A_0$ as defined by
\begin{equation}\label{d:A0}
A_0(U)(s) := \frac1s\int_0^s U(r)\,dr\ .
\end{equation}

%%%%%%%%%%%%%%%%%%%%%%%%%%%%%%%%%%%%%%%%%%%%%%%%%%%%
%%%%%%%%%%%%%%%%%%%%%%%%%%%%%%%%%%%%%%%%%%%%%%%%%%%%

\section{Limit relations at equilibrium}

%%%%%%%%%%%%%%%%%%%%%%%%%%%%%%%%%%%%%%%%%%%%%%%%%%%%
%\subsection{Equilibrium mass-number relation} 

At equilibrium, $f^h=(f^h_i)_{i=1}^\infty$ is constant in time, and the zeroth
and first moments satisfy 
\begin{equation}
\nu_h = m_0(F^h) = \sum_{i=1}^\infty f^h_i \,,
\qquad \mu_h = m_1(F^h) = \sum_{i=1}^\infty ih f^h_i\,.
\end{equation}
Because $f^h$ is an equilibrium solution of Model D, from relations
\qref{eqD:m1def}--\qref{eqD:m0m1}  we find  that these moments are related by
  \begin{equation}\label{eq:m0m1hrel}
{  \frac{\nu_h}{(1-\nu_h)^3} = \frac {\mu_h}h \ .}
  \end{equation}
If we consider the rescaled mass $\mu_h$ as fixed,  the leading behavior of $\nu_h$ as $h\to0$ is given by 
\begin{equation}\label{eq:nuasym}
\nu_h \sim 1- (h/\mu_h)^{1/3}\ .
\end{equation}
In these terms, the tail behavior of the Model D equilibrium from Theorem~\ref{thmD:equilibria} can be recast in the form
\begin{equation}\label{eq:asym1}
 \frac1h f_i^h \sim \frac 1{\mu_h}\,\frac9{8\sqrt\pi} 
\left( \frac{ih}{\mu_h} \right)^{-3/2} 
\left(1 + \frac{4}{27}\frac{h}{\mu_h}\right)^{-i+\frac12} 
\,,\quad i\to\infty.
\end{equation}
As $h\to0$ and $i\to\infty$ with $ih\to x$ and $\mu_h\to\mu$, 
the right-hand side converges to 
\begin{equation}\label{eq:asym2}
\frac1\mu \frac9{8\sqrt\pi} \left(\frac x{\mu}\right)^{-3/2} \exp\left(-\frac4{27} \frac x{\mu}\right)\,.
\end{equation}
This is consistent with the large-size asymptotic behavior of the 
solution of Model C with first moment $\mu$ from Theorem~\ref{thm:cont_CM}.

We have the following  rigorous convergence theorem for the  continuum limit of the discrete equilibria. 

%\[
%\breve f(s)= U_\star(\mu-\mu e^{-s})\,,
%\]
\begin{theorem} Let $f^h$ be a family of equilibria of Model D 
and let  $F^h(dx)=\sum_{i=1}^\infty f^h_i\,\delta_{ih}$. 
If $\mu_h=m_1(F^h)\to\mu$ as $h\to0$, then $F^h$ converges
narrowly to $\Feq(dx)=\feq(x)\,dx$ where 
\[
\feq(x) = \frac{1}{\mu} f_\star\left(\frac x{\mu}\right) \,.
\]
\end{theorem}
\begin{proof} 
The Bernstein transform of the scaled discrete size distribution 
$F^h(dx)=\sum_{i=1}^\infty f^h_i\,\delta_{ih}$ has the following representation,
due to \qref{eqD:usigma}:
%Recall $\breve F(s)$ denotes the equilibrium solution of \qref{eq:BTh}, satisfying
\begin{equation}\label{eqDC:lim1}
\breve F^h(s) = \sum_{i=1}^\infty (1-e^{-sih})f_i^h = 
U_\star\left(\mu_h \frac{1-e^{-sh}}h\right)\,.
\end{equation}
As $h\to0$, from \qref{eqDC:lim1} we have $\breve F^h(s)\to \breve F_{\rm eq}(s)=U_\star(\mu s)$, for every $s\in[0,\infty]$. Now the result follows from
Proposition~\ref{prop:narrow}.
\end{proof}

\begin{figure}[b!t]
    \begin{center}
        \includegraphics[trim=70 320 30 80, clip,width=\textwidth]{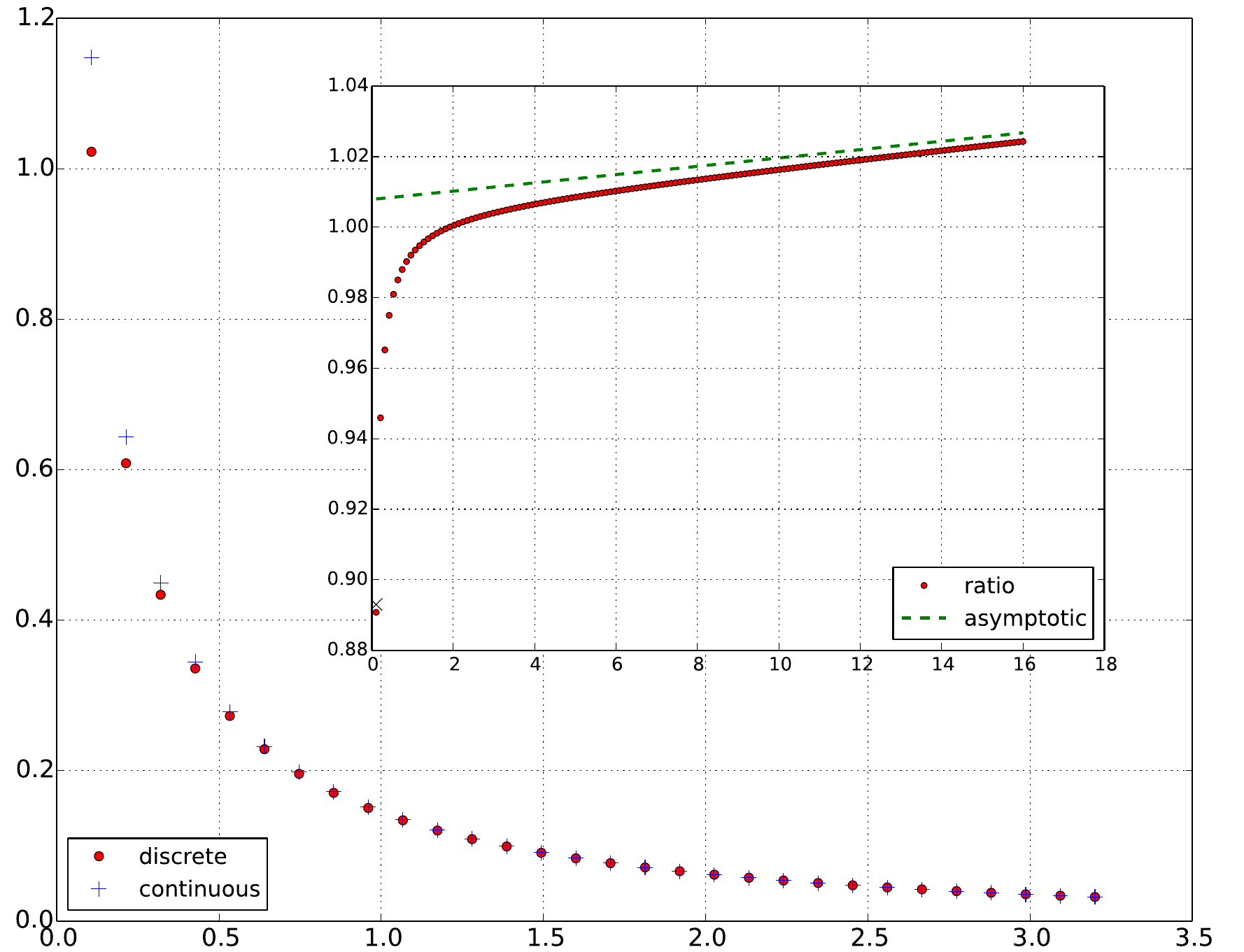}
\caption{Discrete  and continuous profiles: $f^h_i/h$ and $f_\star(ih)$  vs. $ih$ for 
$\nu_h=0.6$, $\mu_h=1$, $h=0.10667$.
Inset: Ratio $( f^h_i/h)/f_\star(ih)$ and asymptotics in \qref{eq:asym3} vs. $ih$.
Cross in inset at $(h,\frac13\Gamma(\frac13))\approx(0.1,0.893)$, see \eqref{eq:h-ratio}.}
    \label{fig:discrete}
    \end{center}
\end{figure}
See Figure~2 for a comparison of the discrete and continuous profiles
$f^h_i/h$ and $f_\star(ih)$, for the parameters $\nu_h=0.6$, $\mu_h=\mu=1$,
corresponding to $h=0.1066\bar6$.
In the inset we compare the ratio of these profiles with the asymptotic expression
coming from \qref{eq:asym1}--\qref{eq:asym2}, 
\begin{equation}\label{eq:asym3}
\frac{f^h_i/h}{f_\star(ih)} \sim  \exp\left(\frac{4ih}{27} \right)
\left(1 + \frac{4h}{27}\right)^{-i+\frac12} \,,
\end{equation}
which reflects the different exponential decay rates for the discrete and
continuous profiles.  The discrete profile also exhibits a transient behavior for
small $i$, starting from the value 
\begin{equation}
\frac1h f^h_1 = \frac{\nu_h-\nu_h^2}{1+2\nu_h} \,\frac{\nu_h}{(1-\nu_h)^3}
\sim \frac13 h^{-2/3}\,,
\end{equation}
coming from \eqref{e:fjrec}, \eqref{eq:m0m1hrel} and \eqref{eq:nuasym}. 
The asymptotic value of the ratio
\begin{equation}\label{eq:h-ratio}
(f^h_1/h)/f_\star(h) \sim \frac13\Gamma\left(\frac13\right)
\end{equation} 
(with $f_\star(h)$ approximated by \eqref{eq:cont_CM_2}-\eqref{eq:cont_CM_3} in Theorem~\ref{thm:cont_CM})
is plotted as a cross at the point $(h,\frac13\Gamma(\frac13))$ in the inset.

%This function
%is related to a steady state solution $U(Cz)$ for equation \qref{eq:Uint1} 
%of the continuous-size model C as follows.  This solution satisfies
%\begin{equation}
%0 = -U(Cz)^2 - U(Cz) + \frac2z \int_0^z U(Cy)\,dy
%\end{equation}
%Transform this via
%\[
%\breve F(s) = U(Cz), \quad 
%z= \frac{1-e^{-sh}}h, \quad y = \frac{1-e^{-rh}}h, \quad dy = e^{-rh}\,dr
%\]
%to find that indeed $\breve F(s)$ is a steady solution of \qref{eq:BTh}.
%Because $U'(0)=1$ the mass of the discrete solution is given by 
%\begin{equation}\label{eq:m0m1h}
% m_1^h = \sum_{n=1}^\infty nh f_n^h = \D_s \breve F(0) = C \ .
%\end{equation}
% Thus we find the mass determines the solution via
% \begin{equation}\label{eq:FUrel}
% \breve F(s) = U(m_1^h z) \ , \qquad z= \frac{1-e^{-sh}}h \ .
% \end{equation}
% The explicit solution is available from \qref{eq:cont_Niwa_equi_3} 
% since $U=\breve f_\star$.

%\subsubsection{Complete monotonicity for model D}
%This will be established based on Corollary 1 of the LP paper.

%%%%%%%%%%%%%%%%%%%%%%%%%%%%%%%%%%%%%%%%%%%%%%%%%%%%
%%%%%%%%%%%%%%%%%%%%%%%%%%%%%%%%%%%%%%%%%%%%%%%%%%%%
\section{Discrete-to-continuum limit}
\label{sec:disc2cont}

We can rigorously prove a weak-convergence result for time-dependent solutions
 of Model D to solutions of Model C, as follows.
\begin{theorem}
Let $F$ be a solution of Model C with initial data $F_0$ a finite measure on $(0,\infty)$,
and let $f^h$, $h\in I$ be solutions of Model D with initial data that satisfy 
$F^h_0\to F_0$ narrowly as $h\to0$.
Then for each $t>0$, we have $F^h_t\to F_t$ narrowly as $h\to0$.
\end{theorem}

\begin{proof}[Proof of the Theorem]
1. The Bernstein functions $U^h=\breve F^h_t$ and $U=\breve F_t$ satisfy
\qref{eq:BTh} and \qref{eqC:Bernstein} respectively.
For $h$ small enough, these functions are 
uniformly bounded globally in time by $C_0=m_0(F_0)+2$, due to the
fact that $m_0(F_0^h)\to m_0(F_0)$ and the bounds 
coming from  \qref{eq:m0evolve} and \qref{eqD:m0bound}. 

2. Next we let
\[
\omega(s,t)=(U-U^h)(s,t), \qquad
\Omega(s,t)=\sup_{0<r\le s} |(U-U^h)(r,t)| \,,
%=\sup_{0<r\le s} |U(r,t)-U^h(r,t)|.
\]
and compute
\begin{equation}\label{eqD:omegat}
\D_t\omega(s,t)=-\omega(U+U^h+1) + 2A_h(\omega)+2(A_0(U)-A_h(U))\,.
\end{equation}
%Now put
%\[
%\Omega(s,t)=\sup_{0<r\le s} |\omega(s,t)| =\sup_{0<r\le s} |U(r,t)-U^h(r,t)|.
%\]
Observe that $|A_h(\omega)(s,t)|\le \Omega(s,t)$, and that
\begin{eqnarray*}
A_0(U)-A_h(U) &=& \frac1s\int_0^s U(r,t)
\left(1- \frac{she^{-rh}}{1-e^{-sh}}\right)\,dr
%\\
%&=& \frac{1}{1-e^{-sh}} \int_0^s U(r) \left(
%\frac{1-e^{-sh}}{sh}-e^{-rh}\right)h\,dr
\end{eqnarray*}
satisfies the bound
\begin{equation}
|A_0(U)-A_h(U)| \le U(s,t) \eta_0(sh), 
\end{equation}
where
\begin{equation}
\eta_0(s) := \sup_{0<r<s} \left|1- \frac{s}{1-e^{-s}} e^{-r} \right| 
\underset{s\to0}\to 0\,.
\end{equation}
%\[
%\omega A_h(\omega)\le \Omega^2
%\]
Now we multiply Eq.~\qref{eqD:omegat} by $2\omega$ and integrate in time
to obtain
\begin{align}
\omega(s,t)^2 -\omega(s,0)^2
& \le  \int_0^t 4\omega(s,\tau) (\Omega(s,\tau) + C_0 \eta_0(sh) )\,d\tau
\nonumber \\
& \le \int_0^t 8 \Omega(s,\tau)^2 \,d\tau +  t C_0^2\eta(sh)^2 \,.
\end{align}
Taking the sup over $s\in[0,\hat s]$ and using a standard Gronwall argument
we infer that for any $s\in(0,\infty)$ and $t>0$,
\begin{equation}
\sup_{0<r<s} |(U-U^h)(r,t)|=\Omega(s,t) \le e^{4t}\left( \Omega(s,0)+\sqrt{t} C_0\eta(sh)\right) \to 0
\label{eqD:gronwall}
\end{equation}
as $h\to0$.

3. We separately study the case $s=\infty$, writing 
\[ 
\hat m^h(t)=m_0(F^h_t)=U^h(\infty,t),\qquad \hat m(t)=m_0(F_t)= U(\infty,t).
\]
Due to \qref{eq:BTh} and \qref{eq:m0evolve}, these moments satisfy
\begin{equation}\label{eqD:hatmt}
\D_t \hat m^h(t) \le -(\hat m^h)^2+\hat m^h \,,
\qquad
 \D_t \hat m(t) = -(\hat m)^2+\hat m \,.
\end{equation}
By the assumption that initial data converge narrowly, we have 
$\hat m^h(0)\to \hat m(0)$, and it is not difficult to  deduce that for each $t>0$,
\[
\limsup_{h\to0} \hat m^h(t) \le \hat m(t)\,.
\]
Now, however, by the monotonicity of $s\mapsto U^h(s,t)$, we have
\[
U(s,t)=\lim_{h\to0} U^h(s,t) \le \liminf_{h\to0} \hat m^h(t)\,.
\]
Because we know from Theorem~\ref{thm:EU-modelC} that 
$\hat m(t)=\lim_{s\to\infty}U(s,t)$, we deduce 
$\hat m^h(t)\to \hat m(t)$ as $h\to0$. 

The narrow convergence $F^h_t\to F_t$ now follows from Proposition~\ref{prop:narrow}.
%\footnote{Check?? Explain??}
\end{proof}

\begin{remark}
Uniqueness for Model C is proved by a simple Gronwall estimate analogous 
to \qref{eqD:gronwall}. We omit details.
\end{remark}

%% END HIDE

%%%%%%%%END HIDE

%% END HIDE

%\input{DLPoutline.tex}

%\vfil\pagebreak

\section*{Acknowledgements}
This material is based upon work supported by the National
Science Foundation under grants 
DMS 1211161, DMS 1515400, and DMS 1514826,
and partially supported by the Center for Nonlinear Analysis (CNA)
under National Science Foundation PIRE Grant no.\ OISE-0967140,
and the NSF Research Network Grant no.\ RNMS11-07444 (KI-Net).
PD acknowledges support from  EPSRC under grant ref: EP/M006883/1, from the
Royal Society and the Wolfson Foundation through a Royal Society Wolfson
Research Merit Award.
% and from NSF by NSF Grant RNMS11-07444 (KI-Net). 
PD is on leave from CNRS, Institut de Math\'ematiques de Toulouse, France. 
JGL and RLP acknowledge support from the 
Institut de Math\'ematiques, Universit\'e Paul Sabatier, Toulouse and the
Department of Mathematics, Imperial College
London, under Nelder Fellowship awards. 

%\footnote{Add Pierre's support.}

%\pagebreak

\bibliographystyle{plain}
\bibliography{Niwabib}

%\vfil\pagebreak
%TO DO:
%\begin{itemize} \setlength{\itemsep}{-3pt}
%\item Add Pierre address and support.
%\item xx Add plots of discrete equilibria in section 11.1 after \eqref{e:fjrec}?
%\item Add title of ref. [11], fix ref. [17]
%\item xx Add abstract.
%\item xx Fix formatting of sectioning style.
%\item xx Finish introduction. Anything else to say?
%\item xx Add graph of $f_\infty=f_\star$, $f_{\rm eq}$ for Model D, in intro?
%\item xx equilibria for other coag kernels and power-law frag rates?
%\item xx Add remark about right endpoint rule in section 14? \\ Guess not - it doesn't conserve mass.
%\item xx Is the $x^{-2/3}$ behavior an anomalous exponent? Does not seem related to equilibrium.
%\end{itemize}
%\vfil

\end{document}